\documentclass[12pt]{amsart}
\usepackage{amsmath,amsfonts,amsthm,amscd,amssymb,mathrsfs,amssymb, multirow, cancel}
\usepackage{stmaryrd}
\usepackage{graphicx}
\newtheorem{theorem}{Theorem}

\newtheorem{proposition}[theorem]{Proposition}
\newtheorem{lemma}[theorem]{Lemma}

\newtheorem{corollary}[theorem]{Corollary}

\newtheorem{remark}[theorem]{Remark}

\newcommand{\CP}{\mathbb{CP}}
\newcommand{\CC}{\mathbb{C}}

\newcommand{\RR}{\mathbb{R}}
\newcommand{\ZZ}{\mathbb{Z}}
\newcommand{\RP}{\mathbb{RP}}
\newcommand{\U}{{\rm{U}}}

\numberwithin{equation}{section}
\numberwithin{theorem}{section}
\numberwithin{table}{section}
\numberwithin{table}{section}
\begin{document}
\bibliographystyle{amsalpha} 
\title[Quotient singularities]{Quotient singularities, eta~invariants, and self-dual metrics}
\author{Michael T. Lock}
\thanks{The first author was partially supported by NSF Grant DMS-1148490}
\address{Department of Mathematics, University of Texas, Austin, 
TX, 78712}
\email{mlock@math.utexas.edu}
\author{Jeff A. Viaclovsky}
\address{Department of Mathematics, University of Wisconsin, Madison, 
WI, 53706}
\email{jeffv@math.wisc.edu}
\thanks{The second author was partially supported by NSF Grant DMS-1405725}
\date{January 13, 2015}
\begin{abstract}
There are three main components to this article:
\vspace{2mm}
\begin{itemize}

\item (i) A formula for the eta invariant of the 
signature complex for any finite subgroup of
${\rm{SO}}(4)$ acting freely on $S^3$ is given. An application of
this is a non-existence result for Ricci-flat ALE metrics on certain 
spaces. 

\vspace{1mm}
\item (ii) A formula for the orbifold correction term 
that arises in the index of the self-dual deformation complex is proved for all finite subgroups of
${\rm{SO}}(4)$ which act freely on $S^3$. Some applications of this
formula to the realm of self-dual and scalar-flat K\"ahler metrics are also discussed. 

\vspace{1mm}
\item  (iii) Two infinite families of scalar-flat anti-self-dual ALE spaces with groups at infinity not contained in ${\rm U}(2)$ are constructed.  Using these spaces,
new examples of self-dual metrics on $n \# \mathbb{CP}^2$ are
obtained for $n \geq 3$.  
\end{itemize}
\end{abstract}
\maketitle
\section{Introduction}
\label{intro}
This focus of this work is on questions arising from the study of four-dimensional spaces that have isolated singularities or non-compact ends which are respectively modeled on neighborhoods of the origin and infinity of $\RR^4/\Gamma$ where $\Gamma\subset{\rm SO}(4)$ is a finite subgroup which act freely on $S^3$.

In particular, we say that $(M^4,g)$ is a {\textit{Riemannian orbifold with isolated singularities}} if $g$ is a smooth Riemannian metric away from a finite set of singular points, and at each singular point the metric is locally the quotient of a smooth $\Gamma$-invariant metric on $B^4$ by some finite subgroup $\Gamma\subset{\rm SO}(4)$ which acts freely on $S^3$.  The group $\Gamma$ is known as the {\textit{orbifold group}} at that point.  

Now, given such a compact orbifold with positive scalar curvature, the Green's function for the conformal Laplacian associated to any point $p\in M$ is guaranteed to exist, which we denote by $G_p$.  Then, the non-compact space $(M^4\setminus\{p\},G_p^2g)$
is a complete scalar-flat orbifold with a coordinate system at infinity arising from inverted normal coordinates in the metric $g$ around $p$.  This, which we call a {\textit{conformal blow-up}},
motivates the following definition.  

We say that a non-compact Riemannian orbifold $(X^4,g)$ is {\textit{asymptotically locally Euclidean (ALE)}} of order $\tau$ if there exists a diffeomorphism
\begin{align}
\psi:X^4\setminus U\rightarrow \big(\RR^4\setminus \overline{B_R(0)}\big)/\Gamma,
\end{align}
where $U\subset X^4$ is compact and $\Gamma\subset {\rm SO}(4)$ is a finite subgroup which acts freely on $S^3$, satisfying $(\psi_*g)_{ij}=\delta_{ij}+\mathcal{O}(r^{-\tau})$ and $\partial^{|k|}(\psi_*g)_{ij}=\mathcal{O}(r^{-\tau-k})$, for any partial derivative of order $k$, as $r\rightarrow \infty$ where $r$ is the distance to some fixed basepoint.  The group $\Gamma$ is known as the {\textit{group at infinity}}.

Now, for such an ALE space, let $u : X^4 \rightarrow \RR^+$ be a function which satisfies $u = O(r^{-2})$ as $r \rightarrow \infty$. 
Then, there is a compactification of the space $(X^4, u^2 g)$ to an orbifold, which we denote by $(\widehat{X^4}, \widehat{g})$.  In general, this will only be a 
$C^{1,\alpha}$-orbifold.  However if the metric satisfies a condition known as anti-self-duality, which we discuss next, there exists a compactification to a $C^{\infty}$-orbifold with positive Yamabe invariant, see \cite{TV2, ChenLeBrunWeber}. 

On an oriented four-dimensional Riemannian manifold $(M,g)$, the Hodge star operator associated to the metric $g$ acting on $2$-forms satisfies $*^2=Id$ and, in turn, induces the decomposition $\Lambda^2=\Lambda^2_+\oplus\Lambda^2_-$, where $\Lambda^2_{\pm}$ are the $\pm 1$ eigenspaces of $*|_{\Lambda^2}$ respectively.  Viewing the Weyl tensor as an operator
$\mathcal{W}_g:\Lambda^2\rightarrow\Lambda^2$, this leads to the decomposition
\begin{align}
\label{Weyl_decomposition}
 \mathcal{W}_g=\mathcal{W}^+_g+\mathcal{W}^-_g,
 \end{align}
 where $\mathcal{W}^{\pm}_g=\Pi^{\pm}\circ\mathcal{W}_g\circ\Pi^{\pm}$, with $\Pi^{\pm}=(Id\pm *)/2$ being the respective projection maps onto $\Lambda^2_{\pm}$.
This decomposition is conformally invariant.  The metric $g$ is called {\textit{self-dual}} if $\mathcal{W}^-_g=0$ and {\textit{anti-self-dual}} if $\mathcal{W}^+_g=0$.  It is important to note that by reversing orientation a self-dual metric becomes anti-self-dual and vice versa.

\begin{remark}
\label{compactification_groups}
{\em 
The conformal compactification of an anti-self-dual ALE space, with group at infinity $\Gamma$, has the same orbifold group at the point of compactification as long as the orientation is reversed, in which case the metric is self-dual.
Therefore, while our focus is on anti-self-dual ALE metrics, 
we will consider the self-dual orientation for compact orbifolds. 
}
\end{remark} 

It is necessary to briefly introduce
the classification of finite subgroups of ${\rm SO}(4)$ which act freely on $S^3$ before stating our main results.  A more thorough discussion is provided in Section \ref{SO(4)_groups}.  These groups are given by the finite subgroups of ${\rm U}(2)$ which act feely on $S^3$, and their orientation reversed conjugates by which we mean that there is an orientation-reversing intertwining map between said groups.  Given a group $\Gamma\subset{\rm SO}(4)$, its orientation reversed conjugate will be denoted by $\overline{\Gamma}\subset{\rm SO}(4)$.  In Table \ref{groups} below, we list of all finite subgroups of \U(2) which act freely on $S^3$.  From this all desired subgroups of ${\rm SO}(4)$ can be understood.  First, a remark on notation:
\begin{itemize}
\item 
For $m$ and $n$ nonzero integers, $L(m,n)$ denotes the cyclic group generated by
\begin{align*}
\left(
\begin{matrix}
\exp(2\pi i/n)&0\\
0&\exp(2\pi i m /n)
\end{matrix}
\right).
\end{align*}

\item 
The map $\phi:S^3\times S^3\rightarrow {\rm SO}(4)$ is the standard double cover, see
\eqref{phi}.

\item 
The binary polyhedral groups are denoted by $D^*_{4n}$, $T^*$, $O^*$, $I^*$ (dihedral, tetrahedral, octahedral, icosahedral respectively).
\end{itemize}

\begin{table}[ht]
\caption{Finite subgroups of \U(2) which act freely on $S^3$}
\label{groups}
\begin{tabular}{ll l}
\hline
$\Gamma\subset{\rm U}(2)$ & Conditions & Order\\
\hline\hline
$L(m,n)$   & $(m,n) = 1$ & $n$\\
$ \phi(L(1,2m) \times D^*_{4n}) $ &  $(m,2n) = 1$ & $4mn$\\
$ \phi(L(1,2m) \times T^*) $  & $(m,6) = 1$ & $24m$\\
$ \phi(L(1,2m) \times O^*) $ & $(m,6) = 1$ & $48m$\\
$ \phi(L(1,2m)\times I^*) $ & $(m,30) = 1$ & $120m$ \\
 Index--$2$ diagonal $\subset   \phi(L(1,4m)\times D^*_{4n})$&  $(m,2) = 2,(m,n)=1$& $4mn$\\
 Index--$3$ diagonal $\subset  \phi(L(1,6m) \times T^*)$  & $(m,6)=3$ & $24m$.\\
\hline
\end{tabular}
\end{table}

\begin{remark}
{\em
The index--$2$ and index--$3$ diagonal subgroups will be denoted by $\mathfrak{I}^2_{m,n}$ and $\mathfrak{I}^3_m$ respectively.  Also, often only non-cyclic subgroups of ${\rm U}(2)$ will be considered.  This excludes the cyclic groups $\phi(L(1,2m) \times D^*_{4})$ and $\mathfrak{I}^2_{m,1}$ (these are the $n=1$ cases).}
\end{remark}

We are now able to state the main results of this work.  Although there is a relationship between the underlying ideas of their proofs, these fall into three distinct categories and are separated accordingly.

\subsection{Eta invariants and Einstein metrics}
\label{Einstein_eta}

Let $(M,g)$ be a compact orbifold having a finite number of isolated singularities $\{p_1,\cdots,p_k\}$ with
corresponding orbifold groups $\{\Gamma_1,\cdots, \Gamma_k\}\subset {\rm SO}(4)$. 
The orbifold signature theorem gives the formula
\begin{align}
\label{tau}
\tau_{top}(M)=\tau_{orb}(M)-\sum_{i=1}^n\eta(S^3/\Gamma_i),
\end{align}
where the quantity $\tau_{orb}(M)$ is the orbifold signature defined by
\begin{align}
\tau_{orb}(M)=\frac{1}{12\pi^2}\int_{M}(||W^+||^2-||W^-||^2)dV_{g},
\end{align}
and $\eta(S^3/\Gamma_i)$ is the $\eta$-invariant of the signature complex.  Since $\Gamma_i\subset {\rm SO}(4)$ is a finite subgroup, this can be shown to be given by
\begin{align}
\label{eta-invariant}
\eta( S^3/\Gamma_i) = \frac{1}{|\Gamma_i|} \sum_{\gamma \neq Id\in \Gamma_i}  \cot \Big(\frac{ r(\gamma)}{2}\Big) 
\cot\Big(\frac{s(\gamma)}{2}\Big),
\end{align}
where $r(\gamma)$ and $s(\gamma)$ denote the rotation numbers of $\gamma\in\Gamma_i$.

The $\eta$-invariants of certain groups are known.  For finite subgroups of ${\rm SU}(2)$ it can be computed directly \cite{nakajima,Hitchin}.  Also, for relatively prime integers $1\leq q<p$, Ashikaga-Ishikawa \cite{AshiIshi} proved a formula for the $\eta$-invariant of the cyclic group $L(q,p)$ in terms of the $e_i$ and $k$ that arise in the modified Euclidean algorithm \eqref{modified_EA}.
(This is also recovered easily from the authors' work in \cite{LockViaclovsky}.)  This formula is given in \eqref{etacyclic}.  Our first results is a formula for all possible cases,  and is proved in Section~\ref{proof_eta_invariant_thm}.
\begin{theorem}
\label{eta_theorem}
Let $\Gamma\subset{\rm SO}(4)$ be a non-cyclic finite subgroup which acts freely on~$S^3$. 
Then, the eta-invariant of the signature complex is given as follows:
\begin{itemize}
\item For $\Gamma\subset{\rm U}(2)$
\begin{align}
\label{eta(U(2))}
\eta(S^3/\Gamma)=\frac{2}{3}\Big( \frac{2m^2+1}{|\Gamma|}\Big)-1+\mathcal{A}_{\Gamma},
\end{align}
where $\mathcal{A}_{\Gamma}$ is a constant given in the following table.
{\em
\begin{table}[ht]
\label{eta_table_intro}
\begin{center}
\renewcommand\arraystretch{1.4}
\begin{tabular}{|l|c|l|}
\hline
$\Gamma\subset {\rm U}(2)$ &  $\mathcal{A}_{\Gamma}$ & Congruences / Conditions\\\hline\hline
$\phi(L(1,2m)\times D^*_{4n})$ &  $\eta(S^3/L(m,n))$ & $(m,2n)=1$\\\hline
$\phi(L(1,2m)\times T^*)$  &$\pm \frac{4}{9}$& $m\equiv \pm 1$ mod $6$\\\hline
\multirow{2}{*}{$\phi(L(1,2m)\times O^*)$} &$\pm\frac{13}{18}$&$m\equiv \pm1$ mod $12$	\\
&$\pm\frac{5}{18}$& $m\equiv \pm5$ mod $12$\\\hline
\multirow{3}{*}{$\phi(L(1,2m)\times I^*)$}  & $\pm\frac{46}{45}$ &$m\equiv \pm1$ mod $30$\\
&$\pm\frac{2}{9}$&$m\equiv \pm7,\pm 13$ mod $30$\\
&$\pm\frac{26}{45}$&$m\equiv \pm11$ mod $30$ \\\hline
$\mathfrak{I}^2_{m,n}$   &$\eta(S^3/L(m,n))$ & $(m,2)=2$, $(m,n)=1$\\\hline
$\mathfrak{I}^3_m$   & $0$& $m\equiv 3$ mod $6$ \\\hline
\end{tabular}
\end{center}
\end{table}
}
\item For $\Gamma\not\subset{\rm U}(2)$
\begin{align}
\eta( S^3/\Gamma)=-\eta( S^3/\overline{\Gamma}),
\end{align}
where $\eta( S^3/\overline{\Gamma})$ is given by \eqref{eta(U(2))} since here $\overline{\Gamma}\subset{\rm U}(2)$.
\end{itemize}
\end{theorem}
\begin{remark}
{\em
Notice that although the $\mathcal{A}_{\Gamma}$ terms for the $\eta$-invariants of $\phi(L(1,2m)\times D^*_{4n})$ and $\mathfrak{I}^2_{m,n}$ contain an $\eta(S^3/L(m,n))$, they can be computed algorithmically by using formula \eqref{etacyclic} for the $\eta$-invariant of a cyclic group.
\em}
\end{remark}

We next give an application of Theorem \ref{eta_theorem}. 
There is a well-known conjecture, due to Bando-Kasue-Nakajima, 
that the only simply-connected 
Ricci-flat ALE metrics in dimension four are the hyperk\"ahler ones \cite{BKN}.
The following shows that the conjecture 
is true, provided one restricts to the diffeomorphism types of
minimal resolutions of $\CC^2 / \Gamma$, or any iterated blow-up thereof..
\begin{theorem}
\label{Ricci_flat_thm}
Let $\Gamma \subset \U(2)$ be a finite subgroup which acts freely on $S^3$,
and let $X$ be diffeomorphic to
the minimal resolution of $\CC^2 / \Gamma$, or any iterated blow-up thereof.  
If $g$ is a Ricci-flat ALE metric on $X$, then $\Gamma \subset {\rm{SU}}(2)$
and $g$ is hyperk\"ahler. 
\end{theorem}
This is proved in Section \ref{proof_Ricci_flat_thm} by applying the Hitchin-Thorpe inequality 
for ALE metrics obtained by Nakajima, together 
with the result in Theorem \ref{eta_theorem}. If one assumes that $g$ is K\"ahler, the
result is trivial, so we emphasize that we only make an assumption about the diffeomorphism type, and do not assume $g$ is K\"ahler.

\subsection{Self-dual deformations}
If $(M,g)$ is a self-dual four-manifold, then the local structure moduli space of self-dual conformal
classes near $g$ is controlled by the following elliptic complex known as  the self-dual 
deformation complex:
\begin{align}
\label{thecomplex}
\Gamma(T^*M) \overset{\mathcal{K}_g}{\longrightarrow} 
\Gamma(S^2_0(T^*M))  \overset{\mathcal{D}_g}{\longrightarrow}
\Gamma(S^2_0(\Lambda^2_-)).
\end{align}
Here $\mathcal{K}_g$ is the conformal Killing operator, $S^2_0(T^*M)$ denotes traceless symmetric $2$-tensors, 
and $\mathcal{D}_g = (\mathcal{W}^-)_g'$ is the linearized anti-self-dual Weyl curvature 
operator.

If $M$ is compact then the index of this complex is given in terms of 
topological quantities via the Atiyah-Singer Index Theorem as
\begin{align}
\label{manifoldindex}
Ind(M, g) =\sum_{i=0}^2(-1)^i \dim( H^i_{SD})=\frac{1}{2} ( 15 \chi_{top}(M) - 29 \tau_{top}(M)),
\end{align}
where $\chi_{top}(M)$ is the Euler characteristic, $\tau_{top}(M)$ is the signature, and $H^i_{SD}$ is the $i^{th}$ cohomology group of \eqref{thecomplex}, for $i = 0,1,2$. 

In \cite{LVsmorgasbord}, the authors discussed the deformation theory of certain scalar-flat K\"ahler ALE metrics.
Unlike the scalar-flat K\"ahler condition, the anti-self-dual 
condition is conformally invariant, so we can transfer the  
the deformation problem of anti-self-dual ALE spaces to their self-dual conformal compactifications.
However, these conformal compactifications are orbifolds, upon which formula \eqref{manifoldindex} does not necessarily hold as there are correction terms required arising from the singularities. Kawasaki proved a version of the Atiyah-Singer index theorem for orbifolds \cite{Kawasaki}, with the orbifold correction terms expressed as 
certain representation-theoretic quantities. In \cite{LockViaclovsky}, the authors 
explicitly determined this correction term in the case of cyclic 
quotient singularities. Our next result is a determination of the correction term for {\textit{all}}  finite subgroups $\Gamma\subset {\rm SO}(4)$ which act freely on $S^3$.
\begin{theorem}
\label{index_theorem}
Let $(M,g)$ be a compact self-dual orbifold with a single orbifold point having orbifold group $\Gamma$, a non-cyclic finite subgroup of ${\rm SO}(4)$.  Then, the index of the self-dual deformation complex on $(M,g)$ is given by
\begin{align*}
Ind(M,g)=\frac{1}{2}(15\chi_{top}(M)-29\tau_{top}(M))+N(\Gamma),
\end{align*}
where $N(\Gamma)$ is given as follows:
\begin{itemize}
\item For $\Gamma\subset{\rm U}(2)$
\begin{align}
\label{N(U(2))}
N(\Gamma)=-4b_{\Gamma}+\mathcal{B}_{\Gamma},
\end{align}
where $-b_{\Gamma}$, given by \eqref{b_Gamma}, is the self-intersection number of the central rational curve in the minimal resolution of $\CC^2/\Gamma$, and $\mathcal{B}_{\Gamma}$ is a constant given
in Table~\ref{index_correction}.  
\item For $\Gamma\not\subset{\rm U}(2)$
\begin{align}
N(\Gamma)=-N(\overline{\Gamma})-\begin{cases}6\phantom{=}&\overline{\Gamma}\subset{\rm SU}(2)\\
7&\overline{\Gamma}\not\subset{\rm SU}(2),
\end{cases}
\end{align}
where $N(\overline{\Gamma})$ is given by \eqref{N(U(2))} since here $\overline{\Gamma}\subset{\rm U}(2)$.
\end{itemize}
\end{theorem}
This is proved in Section \ref{proof_index_theorem}.  Note that Theorem \ref{index_theorem} generalizes easily to the case of orbifolds with any finite number of singularities.
We also note that the second author previously proved such an index formula for the binary polyhedral groups and their orientation reversed conjugates \cite{ViaclovskyIndex}, which is recovered here as well.

In \cite{LVsmorgasbord}, the authors proved the existence of scalar-flat K\"ahler ALE metrics on the minimal resolution of
$\CC^2/\Gamma$ for all finite subgroups $\Gamma\subset{\rm U}(2)$.  When $\Gamma\subset{\rm SU}(2)$ such metrics are necessarily hyperk\"ahler and, from \cite[Section $8$]{LVsmorgasbord},  the anti-self-dual deformations are the same as the scalar-flat K\"ahler deformations.  However, as an application of Theorem \ref{index_theorem}, we have the following, which shows the non-hyperk\"ahler examples of these metrics have many more anti-self-dual deformations than scalar-flat K\"ahler deformations. 
\begin{theorem}
\label{more_ASD_deformations}
Let $g$ be a scalar-flat K\"ahler ALE metric on the 
minimal resolution $X$ of $\CC^2/ \Gamma$, where $\Gamma\subset{\rm U}(2)$ is a finite subgroup which acts freely on $S^3$. 
Then, the dimension of the moduli space of scalar-flat anti-self-dual ALE metrics near $g$ is 
strictly larger than the dimension of the moduli space of scalar-flat K\"ahler
ALE metrics near $g$, unless $\Gamma \subset {\rm{SU}}(2)$ in which case there is equality.
\end{theorem}
This is proved in Section \ref{pf_masdd}, where we also give a formula for the formal dimension of the moduli 
space of scalar-flat anti-self-dual metrics near $g$.

\subsection{Self-dual constructions}
In \cite{ViaclovskyIndex}, the second author posed the question of existence of anti-self-dual ALE spaces with groups at infinity orientation reversed conjugate to the binary polyhedral groups.
The LeBrun negative mass metrics \cite{LeBrunnegative} are examples of scalar-flat K\"ahler, hence anti-self-dual, ALE spaces with groups at infinity orientation reversed conjugate to cyclic subgroups of ${\rm SU}(2)$.  However, since the orientation reversed conjugate groups to the binary polyhedral groups are not contained in ${\rm U}(2)$, there cannot be scalar-flat K\"ahler ALE spaces with these groups at infinity.  Therefore the natural question is that of the existence of anti-self-dual metrics.  This question clearly extends to include the orientation reversed conjugate groups of all non-cyclic finite subgroups of ${\rm U}(2)$ which act freely on $S^3$.  We now give a partial answer to this question and
use this to construct some new examples of self-dual metrics in Corollary \ref{tor_corollary}.

\begin{theorem}
\label{tor}
Let $\Gamma_1=\phi(L(1,2m) \times D^*_{4n})$ and $\Gamma_2=\mathfrak{I}^2_{m,n}$, with integers $m,n$ as specified respectively in Table \ref{groups} so the action on $S^3$ is free.
Then, for $i=1,2$, there exists a scalar-flat anti-self-dual ALE space $(X_i,g_{X_i})$ with group at infinity
$\overline{\Gamma_i}$, the orientation-reversed conjugate group to $\Gamma_i$,
satisfying $\pi_1(X_i) = \ZZ/2\ZZ$. Furthermore, the $g_{X_i}$ may be chosen to admit an isometric $S^1$-action.
\end{theorem}

\begin{remark}{\em
It is still unknown whether there are such 
examples for the other non-cyclic orientation reversed conjugate subgroups which act freely on $S^3$, this is a 
very interesting question. 
}
\end{remark}

Let $\Gamma_1$ and $\Gamma_2$ be as in Theorem \ref{tor}.  For $i=1,2$, let $(Y_i,g_{Y_i})$ denote the scalar-flat K\"ahler, hence anti-self-dual, ALE space with group at infinity $\Gamma_i$,
obtained for the non-cyclic ($n>1$) and cyclic ($n=1$) cases respectively in \cite{LVsmorgasbord,CalderbankSinger}, and let $(X_i,g_{X_i})$ denote the anti-self-dual ALE space with group at infinity $\overline{\Gamma_i}$ obtained in Theorem \ref{tor} above. These can be compactified 
to self-dual orbifolds, $(\widehat{Y_i},\widehat{g}_{Y_i})$ and $(\widehat{X_i},\widehat{g}_{X_i})$, 
and then attached via a self-dual orbifold gluing theorem. 
Although $\pi_1(X_i)=\mathbb{Z}/2\mathbb{Z}$, we will show that $\pi_1( \widehat{X_i}\#\widehat{Y_i} ) = \{1\}$, and thus have the following corollary.
\begin{corollary}
\label{tor_corollary}
Let $\Gamma_1=\phi(L(1,2m) \times D^*_{4n})$ and $\Gamma_2=\mathfrak{I}^2_{m,n}$, with integers $m,n$ as specified in Table \ref{groups} so the action on $S^3$ is free.  For $i=1,2$, define the integer
\begin{align}
\label{tau_connect_sum}
\ell_i(m,n)=3+\begin{cases}
k_{(n-m,n)}+k_{(m-n,m)}\phantom{=}&\text{$n>1$ and $m>1$}\\
m-1\phantom{=}&\text{$n=1$ and $m>1$}\\
n-1\phantom{=}&\text{$n>1$ and $m=1$}\\
0\phantom{=}&\text{$n=1$ and $m=1$},
\end{cases}
\end{align}
where $k_{(q,p)}$ denotes the length of the Hirzebruch-Jung modified Euclidean algorithm for $(q,p)$ (see \eqref{modified_EA} for a description).
Then, for $i=1,2$, 
there exists 
two distinct
sequences of self-dual metrics on $\ell_i(m,n) \# \CP^2$, one limiting to 
an orbifold with a single orbifold point of type $\Gamma_i$, and the other limiting to 
an orbifold with with a single orbifold point of type $\overline{\Gamma_i}$. 
Finally, these examples may be chosen to admit a conformally isometric $S^1$-action. 
\end{corollary}
To the best of the authors' knowledge, 
these are new examples of degeneration of 
self-dual metrics to orbifolds with these orbifold groups.  The proofs of Theorem \ref{tor} and Corollary \ref{tor_corollary} are given in Section~ \ref{sequence_corollary_proof}.

\begin{remark}
\label{l(m,n)}
{\em
The $m=n=1$ case, which can only occur for $\Gamma_1$, is minimal here in the sense that $\ell_1(1,1)=3$ is the smallest number of $\CP^2s$ on which we obtain a self-dual metric by this technique.  Also, notice that for all $\ell>3$ there are multiple possibilities for $m$ and $n$ to obtain a self-dual metric on $\ell\#\CP^2$.  Since each such possibility limits to distinct orbifold metrics, the corresponding metrics on $\ell\#\CP^2$ are themselves distinct. For example, $\ell_1(1,2)\#\CP^2=\ell_2(2,1)\#\CP^2=4\#\CP^2$, but the orbifold limits respectively have singularities of types conjugate to $\phi(L(1,2)\times D^*_8)$ and $\mathfrak{I}^2_{2,1}$.
It is an interesting question whether these self-dual metrics lie in the same or distinct components of the moduli space of self-dual metrics on $\ell\#\CP^2$.
}
\end{remark}

\subsection{Acknowledgements}
The authors would like to express gratitude to Olivier Biquard, Nobuhiro Honda, and Claude LeBrun for many helpful discussions regarding self-dual geometry. 

\section{Background}
\subsection{Group actions and the Hopf fibration}
\label{SO(4)_groups}
It will be convenient to understand ${\rm SO}(4)$ in terms of quaternionic multiplication.  We identity $\CC^2$ with the space of quaternions $\mathbb{H}$ by
\begin{align}
(z_1,z_2)\in\CC^2\longleftrightarrow z_1+z_2\hat{j}\in \mathbb{H},
\end{align}
and consider $S^3$, in the natural way, as the space of unit quaternions.  It is well known that the map
$\phi:S^3\times S^3\rightarrow {\rm SO}(4)$ defined by
\begin{align}
\begin{split}
\label{phi}
\phi(a,b)(h)=a*h*\bar{b},
\end{split}
\end{align}
for $a,b\in S^3$ and $h\in \mathbb{H}$,
is a double cover.
Right multiplication by unit quaternions gives ${\rm SU}(2)$.  Notice that this is just the restriction $\phi|_{1\times S^3}$.  The finite subgroups of ${\rm SU}(2)$ were found in \cite{Coxeter_1940} and 
are given by the cyclic groups $L(-1,n)$, for all integers $n\geq 1$, and the binary polyhedral groups.
Similarly, the restriction
\begin{align}
\label{U(2)_double_cover}
\phi:S^1\times S^3\rightarrow {\rm U}(2),
\end{align}
where the $S^1$ is given by $e^{i\theta}$ for $0\leq \theta <2\pi$, is a double cover of ${\rm U}(2)$.  The finite subgroups of ${\rm U}(2)$ were classified in \cite{DuVal, Coxeter}.  Those which act freely on $S^3$ were later classified in \cite{Scott}.  These are the groups given above in Table \ref{groups}, we refer the reader to \cite{LVsmorgasbord} for an explicit list of generators and a more thorough exposition.

Here, we are interested in all finite subgroups of ${\rm SO}(4)$ which acts freely on $S^3$, not just those in ${\rm U}(2)$.  However, any such group is conjugate, in ${\rm SO}(4)$, to a subgroup of $\phi(S^1\times S^3)$ or $\phi(S^3\times S^1)$, and moreover, these groups themselves are conjugate subgroups of ${\rm O}(4)$, see \cite{Scott}.  Since $\phi(S^1\times S^3)={\rm U}(2)$, we call the subgroups of $\phi(S^3\times S^1)$ the {\textit{orientation reversed conjugate groups}}.  If $\Gamma\subset {\rm U}(2)$, we denote its orientation reversed counterpart by $\overline{\Gamma}\subset\phi(S^3\times S^1)\subset{\rm SO}(4)$.  If $\Gamma\subset {\rm U}(2)$ is finite, then it has a finite set of generators which can be written (not uniquely) as $\{\phi(a_i,b_i)\}_{i=1,\dots,n}$ for some $(a_i,b_i)\in S^1\times S^3$.
Observe that, up to conjugation in ${\rm SO}(4)$, the orientation reversed conjugate group $\overline{\Gamma}$, would be generated by switching the left and right multiplication in the generators, i.e. by the set $\{\phi(b_i,a_i)\}_{i=1,\dots,n}$.

A crucial step that underlies the results of this paper will be to consider quotients of certain orbifolds.  The resulting spaces will have new singular points and it will be essential to understand their orbifold groups.  
To do this, it will necessary to make use of the Hopf fibration.

Given the standard embedding $S^3\subset \CC^2$, and writing $S^2=\CC\cup \{\infty\}$, the Hopf map
$\mathcal{H}:S^3\rightarrow S^2$ is given by
\begin{align}
\label{Hopf_map}
\mathcal{H}(z_1,z_2)\mapsto z_1/z_2.
\end{align}
Observe that the Hopf fiber, over a general $z\in S^2$, is the $S^1$ given by
\begin{align}
\label{H_fiber}
e^{i\theta}(|z|^2+1)^{-1/2}(z,1)=e^{i\theta}(|z|^2+1)^{-1/2}(z+\hat{j})\in S^3.
\end{align} 
Using \eqref{H_fiber} to examine this fibration under quaternionic multiplication, one finds the following.  The Hopf fibration is preserved by all right multiplication, however it is only preserved under left multiplication by quaternions of the form $e^{i\theta}$ and $e^{i\theta}*\hat{j}$.  
Thus, from \eqref{U(2)_double_cover}, it is clear that all of ${\rm U}(2)$ preserves the Hopf fiber structure.  To find all other finite subgroups of ${\rm SO}(4)$ which act freely on $S^3$ and preserve the Hopf fibration, it is only necessary to examine those which are orientation reversed conjugate
to the finite subgroups of ${\rm U}(2)$ listed in Table \ref{groups}.  Since the orientation reversed conjugate groups are generated by switching left and right quaternionic multiplication of the generators of subgroups of $\U(2)$, it will be precisely those orientation reversed conjugate to subgroups of $\U(2)$ that have only elements of the form $e^{i\theta}$ or $e^{i\theta}*\hat{j}$ acting on the right which preserve the Hopf fibration.
In \cite{LVsmorgasbord}, the authors provided a list of these groups along with their generators, and referring to this it is clear that the only non-cyclic orientation reversed conjugate groups that preserve the Hopf fiber structure are those given in Table \ref{orientation_reversed_groups}.
\begin{table}[ht]
\caption{Groups preserving the Hopf fibration but not contained in ${\rm{U}}(2)$.}
\label{orientation_reversed_groups}
\begin{center}
\renewcommand\arraystretch{1.4}
\begin{tabular}{lll}
\hline
Non-cyclic $\overline{\Gamma}\subset {\rm SO}(4)$ & Conditions    &  Generators\\\hline\hline
$\overline{\phi(L(1,2m)\times D^*_{4n})}$  & $(m,2n)=1$       &$ \phi(1,e^{\frac{\pi i}{m}})$, $\phi(e^{\frac{\pi i}{n}},1)$, $ \phi(\hat{j},1)$\\
$ \overline{\mathfrak{I}^2_{m,n}}$ &  $(m,2)=2, (m,n)=1$      &$ \phi(1,e^{\frac{\pi i}{m}})$, $\phi(e^{\frac{\pi i}{n}},1)$, $\phi(\hat{j},e^{\frac{\pi i}{2m}})$\\
\hline
\end{tabular}
\end{center}
\end{table}

\subsection{Minimal resolutions and ALE metrics}
\label{minimal_resolutions}
For all finite subgroups $\Gamma\subset \U(2)$, which act freely on $S^3$, 
Brieskorn described the minimal resolution of $\CC^2/\Gamma$ complex analytically \cite{Brieskorn}.  
Here, we will first describe the minimal resolution of cyclic quotients, and then use this to to provide a description for all possible cases.

Let $\tilde{X}$ be the minimal resolution 
of $\CC^2/ L(q,p)$, where $1 \leq q < p$ are relative prime integers. The exceptional divisor of $\tilde{X}$, known as a Hirzebruch-Jung string, has intersection matrix:
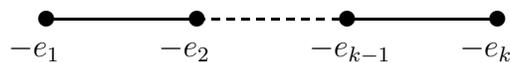
\begin{figure}[h]
\setlength{\unitlength}{2cm}
\begin{picture}(7,0.1)(-1,0)
\linethickness{.3mm}
\put(1,0){\line(1,0){1}}
\put(1,0){\circle*{.1}}
\put(.75,-.25){$-e_1$}
\put(2,0){\circle*{.1}}
\put(1.75, -.25){$-e_2$}
\multiput(2,0)(0.1,0){10}
{\line(1,0){0.05}}
\put(3,0){\circle*{.1}}
\put(2.75, -.25){$-e_{k-1}$}
\put(3,0){\line(1,0){1}}
\put(4,0){\circle*{.1}}
\put(3.75, -.25){$-e_k$}
\end{picture}
\vspace{5mm}
\label{cycfig}
\caption{Hirzebruch-Jung string.}
\end{figure}

\noindent
where the $e_i$ and $k$ are determined by the following Hirzebruch-Jung modified Euclidean algorithm:
\begin{align}
\begin{split}
\label{modified_EA}
p&=e_1q-a_1\\
q&=e_2a_1-a_2\\
&\hspace{2mm} \vdots \\
a_{k-3}&=e_{k-1}a_{k-2}-1\\
a_{k-2} &= e_k a_{k-1} = e_k,
\end{split}
\end{align}
where the self-intersecton numbers $e_i \geq 2$ and $0 \leq a_i < a_{i-1}$, see \cite{Hirzebruch1953}.  The integer $k$ is called the {\textit{length}} of the modified Euclidean algorithm.  (These values can also be understood in terms of a continued fraction expansion of $q/p$.) 

\begin{remark}{\em
We will often need to distinguish the length of the modified Euclidean algorithm for a particular pair of
relatively prime integers $1 \leq q <p$, and therefore we denote this by $k_{(q,p)}$ when necessary.
}
\end{remark}

For a non-cyclic finite subgroup $\Gamma\subset\U(2)$ which acts freely on $S^3$, the exceptional divisor of the minimal resolution of $\CC^2/\Gamma$ is a tree of rational curves with normal crossing singularities.  There is a distinguished rational curve which intersects exactly three other rational curves.  We refer to this as the \textit{central rational curve}, and it has self-intersection number
\begin{align}
\label{b_Gamma}
-b_{\Gamma}=-2-\frac{4m}{|\Gamma|}\Big[m-\Big(m\text{ mod } \frac{|\Gamma|}{4m}\Big)\Big],
\end{align}
where $m$ corresponds to that of the group as in Table \ref{groups}.  Emanating from the central rational curve are three Hirzebruch-Jung strings corresponding to the singularities $L(p_i-q_i,p_i)=\overline{L(q_i,p_i)}$, for $i=1,2,3$, where the $L(q_i,p_i)$ are the cyclic singularities of the orbifold quotients of Theorem \ref{quotient_theorem} specified for each group below.

In \cite{LVsmorgasbord}, the authors constructed scalar-flat K\"ahler ALE metrics on the minimal resolution of $\CC^2/\Gamma$ for all non-cyclic finite subgroups $\Gamma\subset{\rm U}(2)$ which act freely on $S^3$.   The previously known examples of such spaces are for non-cyclic finite subgroups $\Gamma\subset{\rm SU}(2)$ (these are the binary polyhedral groups), for which  Kronheimer obtained and classified hyperk\"ahler metrics on these minimal resolutions \cite{Kronheimer,Kronheimer2}.

LeBrun constructed a $\U(2)$-invariant scalar-flat K\"ahler ALE metric on the minimal resolution of $\CC^2/L(1,\ell)$ for all positive integers $\ell$ \cite{LeBrunnegative}.  The $\ell=2$ case is the well-known Eguchi-Hanson metric \cite{EguchiHanson}.  The minimal resolution here is the total space of the bundle $\mathcal{O}(-\ell)$ over $\CP^1$.  These are known as the LeBrun negative mass metrics and are denoted by $(\mathcal{O}(-\ell),g_{LB})$.

In \cite{CalderbankSinger}, Calderbank-Singer used the Joyce ansatz \cite{Joyce1995} to explicitly construct toric scalar-flat K\"ahler metrics on the minimal resolution of $\CC^2/L(q,p)$ for all relatively prime integers $1\leq q<p$.  When $q=1$ these are the LeBrun negative mass metrics, and when $q=p-1$ these are the toric multi-Eguchi-Hanson metrics (all monopole points lie on a common line), see \cite{GibbonsHawking}.

\subsection{Orbifold quotients}
\label{orbifold_quotients}

It will be essential to the work here to examine quotients of certain weighted projective spaces.  These are defined in general as follows.  For relatively prime integers $1 \leq r \leq q \leq p$, the \textit{weighted projective space} $\mathbb{CP}^2_{(r,q,p)}$ is $S^{5}/S^1$, 
where $S^1$ acts by
\begin{align*}
(z_0,z_1,z_2)\mapsto (e^{ir\theta}z_0,e^{iq\theta}z_1 ,e^{ip\theta}z_2),
\end{align*}
for $0\leq \theta <2\pi$.
Each weighted projective space is a complex orbifold which admits a unique Bochner-K\"ahler metric that we refer to as the \textit{canonical Bochner-K\"ahler metric}, see \cite{Bryant} for existence and \cite{DavidGauduchon} for uniqueness.  In real four-dimensions the Bochner-tensor is precisely the anti-self-dual part of the the Weyl tensor, so these metrics are self-dual K\"ahler.

Topologically, the conformal compactification of $\mathcal{O}(-n)$ is the weighted projective space $\CP^2_{(1,1,n)}$, which has a singularity of type $L(-1,n)$ at $[0,0,1]$, the point of compactification.  In \cite{DabkowskiLock}, the first author and Dabkowski proved an explicit K\"ahler conformal compactification of $\U(n)$-invariant K\"ahler ALE spaces, i.e. the conformal compactification of such spaces to K\"ahler orbifolds.  (It is important to note that the resulting spaces are K\"ahler with respect to reverse-oriented complex structures.)  Therefore, the LeBrun negative mass metric on $\mathcal{O}(-n)$ has a conformal compactification to a K\"ahler metric on $\widehat{\mathcal{O}}(-n)=\CP^2_{(1,1,n)}$.  Moreover,
 since the ALE metric is anti-self-dual the compactified metric is self-dual, so this is necessarily the canonical Bochner-K\"ahler metric on $\CP^2_{(1,1,n)}$.  Previously, in \cite{Joyce1991}, Joyce proved that there is a quaternionic metric on 
$\CP^2_{(1,1,n)}$ which must be conformal to $(\mathcal{O}(-n),g_{LB})$.  However, from
\cite{DabkowskiLock}, we see that the canonical Bochner-K\"ahler metric on $\CP^2_{(1,1,n)}$ is given explicitly by
\begin{align}
g_{BK}=\frac{dr^2}{(1+r^2)(1+nr^2)}+\frac{r^2}{r^2+1}\Big[\sigma_1^2+\sigma_2^2+\Big(\frac{1+nr^2}{r^2(1+r^2)}\Big)\sigma_3^2\Big],
\end{align}
where $\sigma_1,\sigma_2,\sigma_3$ is the usual left-invariant coframe on ${\rm SU}(2)$ and $r=0$ corresponds to $[0,0,1]$, the point of compactification.  Also, the rational curve defined by
\begin{align}
\Sigma:= \{[z_0,z_1,0]: z_0, z_1 \in \CC\} \subset \CP^2_{(1,1,n)}
\end{align}
will frequently be considered, so we make a point of distinguishing it here.  The corresponding rational curve in quotients of $\CP^2_{(1,1,n)}$ will be denoted by $\Sigma$ as well.

In \cite{LVsmorgasbord}, for each non-cyclic finite subgroup $\Gamma\subset{\rm U}(2)$ which acts freely on $S^3$, 
the authors took a specific quotient of a certain $(\mathcal{O}(-2m),g_{LB})$ to obtain a scalar-flat K\"ahler ALE orbifold with group at infinity $\Gamma$ and with all singularities isolated and of cyclic type.
The conformal compactification factor from \cite{DabkowskiLock} descends to compactify these quotients to self-dual K\"ahler orbifold quotients of $(\CP^2_{(1,1,2m)},g_{BK})$.  The following theorem is an immediate consequence of \cite[Theorem 4.1]{LVsmorgasbord}.
\begin{theorem}
\label{quotient_theorem}
For each non-cyclic finite subgroup $\Gamma\subset{\rm U}(2)$ which acts freely on $S^3$, the quotient
$(\CP^2_{(1,1,2m)},g_{BK})/\Gamma$ is a self-dual K\"ahler orbifold with four isolated singularities -- one at the point of compactification with orbifold group $\Gamma$, and three on the rational curve $\Sigma$ with orbifold groups specified precisely as follows.
{\em
\begin{align*}
\begin{array}{lllll}
\hline
\Gamma\subset {\rm U}(2)& \text{Conditions}  &  &\hspace{-3mm}\text{Orbifold}& \hspace{-3mm}\text{groups} \\\hline\hline
\phi(L(1,2m)\times D^*_{4n}) & (m,2n)=1 & L(1,2)& L(1,2)& L(m,n)\\  
\phi(L(1,2m)\times T^*)  & (m,6)=1& L(1,2) &  L(m,3) & L(m,3)\\ 
\phi(L(1,2m)\times O^*) & (m,6)=1 & L(1,2)& L(m,3) & L(m,4)\\
\phi(L(1,2m)\times I^*)  & (m,30)=1& L(1,2)& L(m,3) & L(m,5) \\
\text{$\mathfrak{I}^2_{m,n}$} &(m,2)=2,(m,n)=1& L(1,2) &L(1,2)&L(m,n) \\
 \text{$\mathfrak{I}^3_m$}& (m,6)=3 & L(1,2)& L(1,3) &L(2,3)\\\hline
\end{array}
\end{align*}
\em}
\end{theorem}

\begin{remark}
{\em
We write the orbifold groups of the three singularities on the rational curve $\Sigma$ of the quotients
$(\CP^2_{(1,1,2m)},g_{BK})/\Gamma$ in Theorem \ref{quotient_theorem} as
$L(q_i,p_i)$, for $i=1,2,3$,
where $q_i$ is chosen modulo $p_i$ to satisfy $1\leq q_i< p_i$.
}
\end{remark}

\section{Eta invariants and Einstein metrics}
\label{remarks_on_eta}
In this section, we first derive the general formula for the 
$\eta$-invariant given in Theorem \ref{eta_theorem}, and then prove the non-existence 
result for Einstein metrics stated in Theorem \ref{Ricci_flat_thm}. 

\subsection{Proof of Theorem \ref{eta_theorem}}
\label{proof_eta_invariant_thm}
Recall that, up to conjugation in ${\rm SO}(4)$, the finite subgroups of ${\rm SO}(4)$ which act freely on $S^3$ are the finite subgroups of ${\rm U}(2)$, which act freely on $S^3$, and their orientation-reversed conjugates \cite{Scott}.  

We first discuss the $\eta$-invariant of cyclic groups.  For relatively prime integers $1\leq q<p$, the following formula is proved in \cite{AshiIshi,LockViaclovsky}:
\begin{align}
\label{etacyclic}
\eta(S^3/L(q,p))=\frac{1}{3}\Big(\sum_{i=1}^ke_i+\frac{q^{-1;p}+q}{p}\Big)-k,
\end{align}
where the $e_i$ and $k$ are as defined in \eqref{modified_EA}, and $q^{-1;p}$ denotes the inverse of $q\text{ mod } p$.

\begin{remark}
\label{eta_proof_rmk}
{\em
 Let $\Gamma\subset{\rm U}(2)$ be a finite subgroup which acts freely on $S^3$, and let $\overline{\Gamma}\subset{\rm SO}(4)$ denote its orientation reversed conjugate.  It is elementary that 
 $\eta(S^3/\Gamma)=-\eta(S^3/\overline{\Gamma})$.
Therefore, since all possible cyclic groups are orientation preserving conjugate to some $L(q,p)$, and the $\eta$-invariants of these are given by \eqref{etacyclic}, the theorem will follow from finding the $\eta$-invariant of all non-cyclic finite subgroups of ${\rm U}(2)$ which act freely on $S^3$.
}
\end{remark}

Since $\tau_{top}(\CP^2_{(1,1,2m)},g_{BK})=1$, from \eqref{tau} and \eqref{etacyclic}, observe that the orbifold signature is
\begin{align}
\tau_{orb}(\CP^2_{(1,1,2m)})&=\eta(1,2m)+1=\frac{2m^2+1}{3m}.
\end{align}
Let $\Gamma\subset{\rm U}(2)$ be a finite subgroup which acts freely on $S^3$ and consider
\begin{align}
\widetilde{\Gamma}=\Gamma/L(1,2m)\subset {\rm U}(2)/L(1,2m).
\end{align}
Clearly, $\widetilde{\Gamma}$ acts effectively on $\CP^2_{(1,1,2m)}$ and $\CP^2_{(1,1,2m)}/\widetilde{\Gamma}=\CP^2_{(1,1,2m)}/\Gamma$ is an orbifold with orbifold groups $\Gamma$ at the point of compactification and $L(q_i,p_i)$, for $i=1,2,3$, on the rational curve $\Sigma$ as
specified in Theorem \ref{quotient_theorem}.
From the orbifold signature theorem \eqref{tau}, since $\tau_{top}(\CP^2_{(1,1,2m)}/\widetilde{\Gamma})=1$ and $\tau_{orb}(\CP^2_{(1,1,2m)}/\widetilde{\Gamma})=\frac{1}{|\widetilde{\Gamma}|}\Big(\frac{2m^2+1}{3m}\Big)$, we find that
\begin{align}
\label{eta_general}
\eta(S^3/\Gamma)=\frac{1}{|\widetilde{\Gamma}|}\Big(\frac{2m^2+1}{3m}\Big)-1-\sum_{i=1}^3\eta(L(q_i,p_i)).
\end{align}
The proof is completed for each case by using the appropriate singularities as specified in Theorem \ref{quotient_theorem}, and then computing the corresponding cyclic eta-invariants for the particular congruences of $m$.  

For instance, when $\Gamma=\phi(L(1,2m)\times T^*)$, observe from Theorem \ref{quotient_theorem} that the singularities are of types $L(1,2), L(m,3), L(m,3)$.  Therefore
\begin{align*}
\eta(S^3/\phi(L(1,2m)\times T^*))=&\frac{1}{12}\Big(\frac{2m^2+1}{3m}\Big)-1-\eta(S^3/L(1,2))-2\eta(S^3/L(m,3))\\
=&\frac{1}{12}\Big(\frac{2m^2+1}{3m}\Big)-1\pm \frac{4}{9}\phantom{=}\text{for $m\equiv \pm 1\text{ mod }6$}.
\end{align*}
The idea for the other cases is identical and the computations, which follow similarly, are omitted.

\begin{remark}
{\em
The $\eta$-invariants of the binary polyhedral groups were known \cite{nakajima}.  However, their direct computation is arduous and here they are recovered simply.
}
\end{remark}
\subsection{Proof of Theorem \ref{Ricci_flat_thm}}
\label{proof_Ricci_flat_thm}
We will use Theorem \ref{eta_theorem} along with the following ALE analogue of the Hitchin-Thorpe inequality due to Nakajima.  Let $(M,g)$ be a Ricci-flat  ALE manifold with group at infinity $\Gamma$, then
\begin{align}
\label{ALE_HT}
2\Big(\chi_{top}(M)-\frac{1}{|\Gamma|}\Big)\geq 3|\tau_{top}(M)-\eta(S^3/\Gamma)|,
\end{align}
with equality if and only if $W^+$ or $W^-$ vanishes identically \cite{nakajima},
see also \cite{Hitchin}.

First, let $X$ be diffeomorphic to the minimal resolution of $\CC^2/\Gamma$ for some finite $\Gamma\subset{\rm U}(2)$ which acts freely on $S^3$.  For relatively prime integers $1\leq q<p$, let $k_{(q,p)}$ denote the length, and $e^{(q,p)}_j$ the coefficients, of the modified Euclidean algorithm \eqref{modified_EA} for $(q,p)$.  Then
\begin{align}
\begin{split}
\label{chi_tau_Ric_flat}
\chi_{top}(X)&=1-\tau_{top}(X)\\
\tau_{top}(X)&=\begin{cases}
\displaystyle-1-\sum_{i=1}^3k_{(p_i-q_i,p_i)}\phantom{=}&\text{$\Gamma$ non-cyclic}\\
\displaystyle-k_{(q,p)}\phantom{=}&\Gamma=L(q,p),
\end{cases}
\end{split}
\end{align}
where the $L(p_i-q_i,p_i)$, for $i=1,2,3$, are orientation reversed orbifold groups of those given for $\Gamma$ in Theorem \ref{quotient_theorem}.  These orbifold groups are orientation reversed conjugate to those on the quotient of weighted projective space appearing in the proof of Theorem \ref{eta_theorem}, so rewriting \eqref{eta_general} with respect to these groups here just reverses the sign of the $\eta$-invariants of the cyclic groups appearing in the sum.  

For non-cyclic $\Gamma$ there is the term $\sum_{i=1}^3k_{(p_i-q_i,p_i)}$ in both $\tau_{top}(X)$ and $\eta(S^3/\Gamma)$.  This follows from 
 \eqref{etacyclic}, \eqref{eta_general}, and \eqref{chi_tau_Ric_flat}.  Therefore, we compute that
\begin{align*}
\tau_{top}(X)-\eta(S^3/\Gamma)=&-\frac{1}{|\widetilde{\Gamma}|}\Big(\frac{2m^2+1}{3m}\Big)-\frac{1}{3}\sum_{i=1}^3\Big(\sum_{j=1}^{k_{(p_i-q_i,p_i)}}e^{(p_i-q_i,p_i)}_j\Big)\\
&-\frac{1}{3}\sum_{i=1}^3\frac{(p_i-q_i)^{-1;p_i}+(p_i-q_i)}{p_i}.
\end{align*}
Since each summand here is positive, rewrite \eqref{ALE_HT} as
\begin{align}
\begin{split}
\label{ALE_HT_min}
2\sum_{i=1}^3k_{(p_i-q_i,p_i)}+4-\frac{2}{|\Gamma|}\geq& \frac{1}{|\widetilde{\Gamma}|}\Big(\frac{2m^2+1}{m}\Big)+\sum_{i=1}^3\Big(\sum_{j=1}^{k_{(p_i-q_i,p_i)}}e^{(p_i-q_i,p_i)}_j\Big)\\
&+\sum_{i=1}^3\frac{(p_i-q_i)^{-1;p_i}+(p_i-q_i)}{p_i}.
\end{split}
\end{align}
Observe that for any non-cyclic $\Gamma\subset{\rm SU}(2)$ equality holds in \eqref{ALE_HT_min}, so any Ricci-flat metric in this case must also be anti-self-dual.  Therefore, these are the hyperk\"ahler ALE metrics classified by Kronheimer \cite{Kronheimer, Kronheimer2}.  Now, we will show that the inequality \eqref{ALE_HT_min} is violated for all non-cyclic $\Gamma\not\subset{\rm SU}(2)$.  Recall, from \eqref{modified_EA}, that all $e^{(p_i-q_i,p_i)}_j\geq 2$.  Since $\Gamma\not\subset{\rm SU}(2)$ is non-cyclic, for each $L(p_i-q_i,p_i)$ there is at least one $e^{(p_i-q_i,p_i)}_j>2$, so the problem reduces to proving that the inequality
\begin{align}
\label{ALE_HT_min_cyclic}
1-\frac{2}{|\Gamma|}\geq \frac{1}{|\widetilde{\Gamma}|}\Big(\frac{2m^2+1}{m}\Big)+\sum_{i=1}^3 \frac{(p_i-q_i)^{-1;p_i}+(p_i-q_i)}{p_i}
\end{align}
is violated.  Finally, observe that this is so since for all such groups at least one of the $L(p_i-q_i,p_i)=L(1,2)$ for which the term $\frac{(p_i-q_i)^{-1;p_i}+(p_i-q_i)}{p_i}=1$.

For $\Gamma=L(q,p)$ cyclic, \eqref{ALE_HT} reduces to
\begin{align}
2\Big(k_{(q,p)}+1-\frac{1}{p}\Big)\geq \sum_{j=1}^{k_{(q,p)}}e_j+\frac{q^{-1;p}+q}{p}.
\end{align}
Clearly, this inequality holds if and only if $\Gamma=L(-1,p)$, in which case it holds with equality, so any Ricci-flat metric in this case must also be anti-self-dual.  Therefore, these are the hyperk\"ahler Gibbons-Hawking multi-Eguchi-Hanson metrics.

Finally, let $X$ be diffeomorphic to the iterated blow-up of the minimal resolution of $\CC^2/\Gamma$.  
Then, here \eqref{ALE_HT} reduces to the inequalities \eqref{ALE_HT_min} and \eqref{ALE_HT_min_cyclic} with $2\ell$ and $3\ell$ added to the left and right hand sides of each respectively, where $\ell\geq 1$ is the number of blow-ups.  Given the previous arguments, it is easy to see that \eqref{ALE_HT} is always violated, and therefore no Ricci-flat metrics can exist.

\section{Self-dual deformations}
\label{SDD}
In Section \ref{proof_index_theorem} the proof of Theorem \ref{index_theorem} is given.  Then, using this  along with the previous work of the authors in \cite{LockViaclovsky}, we prove Theorem \ref{more_ASD_deformations} in Section \ref{pf_masdd}.

\subsection{Proof of Theorem \ref{index_theorem}}
\label{proof_index_theorem}
In \cite{LockViaclovsky}, the authors proved an index formula for the anti-self-dual deformation complex on an orbifold with isolated cyclic singularities, which is easily adjusted to find an index formula the self-dual deformation complex.
As an intermediate step to this, we showed that if $(M,g)$ is a compact self-dual orbifold with finitely many isolated singularities $\{p_1,\cdots, p_n\}$ having corresponding orbifold groups $\{\Gamma_1,\cdots, \Gamma_n\}$, where $\Gamma_i\subset {\rm SO}(4)$ is any finite subgroup which acts freely on $S^3$, then the index of the self-dual deformation complex can be expressed by a topological quantity and a correction term depending only on the $\Gamma_i$ as follows 
\begin{align}
\label{index_quantities}
Ind(M,g)=\frac{1}{2} ( 15 \chi(M) - 29 \tau(M))+\sum_{i=1}^nN(\Gamma_i).
\end{align} 
For $\Gamma=L(q,p)$ the authors proved that
\begin{align}
\label{N(L(q,p))}
N(L(q,p))=\begin{cases}
\displaystyle-\sum_{i=1}^k4e_i+12k-10\phantom{=}&\text{when $q>1$}\\
\displaystyle -4p+4\phantom{=}&\text{when $q= 1$},
\end{cases}
\end{align}
where $k$ and the $e_i$ are as in the modified Euclidean algorithm \eqref{modified_EA}.  

Recall that, up to conjugation in $\rm{SO}(4)$, the set of finite subgroups of $\rm{SO}(4)$ which act freely on $S^3$ are given by the finite subgroups of ${\rm U}(2)$ which act freely on $S^3$ and their orientation reversed conjugates \cite{Scott}.
From the following lemma, we see that it will be enough to find the correction terms for those subgroups of ${\rm U}(2)$.
\begin{lemma}
\label{SD_index_lemma}
Let $\Gamma\subset{\rm U}(2)$ be a non-cyclic finite subgroup which acts freely on $S^3$, and let $\overline{\Gamma}\subset{\rm SO}(4)$ denote its orientation reversed conjugate group.  Then, the self-dual index correction term for $\overline{\Gamma}$ is given in terms of that for $\Gamma$ by
\begin{align*}
N(\overline{\Gamma})=-N(\Gamma)-\begin{cases}
7\phantom{=}&\Gamma\not\subset{\rm SU}(2)\\
5\phantom{=}&\Gamma\subset{\rm SU}(2).
\end{cases}
\end{align*}
\end{lemma}
\begin{proof}
Consider the quotient $(S^4,g)/\Gamma$, where $g$ is the standard round metric and $\Gamma$ acts around the $x_5$-axis.  This is a self-dual orbifold with two singularities, one of type $\Gamma$ and the other of type $\overline{\Gamma}$.  It is well known that both $H^1_{SD}$ and $H^2_{SD}$ of the self-dual deformation complex vanish in this case, thus the index is given by $\dim(H^0_{SD})$.  From \cite{MCC}, we have that $\dim(H^0_{SD})=1$ if $\Gamma\not\subset{\rm SU}(2)$ and not cyclic, and $\dim(H^0_{SD})=3$ if $\Gamma\subset{\rm SU}(2)$ and not cyclic.  Equating this with the index obtained from \eqref{index_quantities}, since $\chi_{top}(S^4/\Gamma)=2$ and $\tau_{top}(S^4/\Gamma)=0$, we find that
\begin{align}
8+N(\Gamma)+N(\overline{\Gamma})=\begin{cases}
1\phantom{=}&\Gamma\not\subset{\rm SU}(2)\\
2\phantom{=}&\Gamma\subset{\rm SU}(2),
\end{cases}
\end{align}
from which we can solve for $N(\overline{\Gamma})$ to complete the proof.
\end{proof}

Therefore, to complete the proof, it is only necessary to find the correction term for finite subgroups of ${\rm U}(2)$ which act freely on $S^3$, and the plan for the remainder of the proof is as follows.  
For any finite subgroup $\Gamma\subset {\rm U}(2)$ which acts freely on $S^3$, recall the orbifold quotients $(\CP^2_{(1,1,2m)},g_{BK})/\Gamma$ from Theorem \ref{quotient_theorem}.
In \cite{HondaOn}, Honda discovered the explicit form of the ${\rm U}(2)$-action on $H^1_{SD}$ of the self-dual deformation complex of $(\CP^2_{(1,1,2m)},g_{BK})$.  Applying a representation theoretic argument to this, we find the dimension of the space of invariant elements of $H^1_{SD}$ under the quotient by $\Gamma$.  Since finding $\dim(H^1_{SD})$ of the self-dual deformation complex on the quotient $(\CP^2_{(1,1,2m)},g_{BK})/\Gamma$ is equivalent to finding the dimension of space of invariant elements of $H^1_{SD}$ on $(\CP^2_{(1,1,2m)},g_{BK})$ under the action of $\Gamma$, we then use this to recover the index.  Finally we solve for $N(\Gamma)$ in terms of 
the index, which at this point is known, and the correction terms of the cyclic quotient singularities, which are known from \eqref{N(L(q,p))}, that arise in the quotient.

In the following proposition we find the dimension of $H^1_{SD}$ of $(\CP^2_{(1,1,2m)},g_{BK})/\Gamma$, which we are able to give simply in terms of $b_{\Gamma}$ (the negative of the self-intersection number of the central rational curve).

\begin{proposition}
\label{invariant_proposition}
Let $\Gamma\subset{\rm U}(2)$ be a non-cyclic finite subgroup which acts freely on $S^3$.  Then, the dimension of the space of the first cohomology group of the self-dual
deformation complex on $(\CP^2_{(1,1,2m)},g_{BK})/\Gamma$ is given by
\begin{align*}
\dim(H^1_{SD,\Gamma})=4b_{\Gamma}-\mathcal{C}_{\Gamma}=\frac{16m}{|\Gamma|}\Big[m-\Big(m\text{ {\em mod} }\frac{|\Gamma|}{4m}\Big)\Big]+8-\mathcal{C}_{\Gamma},
\end{align*}
where $\mathcal{C}_{\Gamma}$ is a constant given by
{\em
\begin{align*}
\renewcommand\arraystretch{1.4}
\begin{array}{|l|c|l|}
\hline
\Gamma\subset {\rm U}(2) &\phantom{=}  \mathcal{C}_{\Gamma} \phantom{=}& \text{Congruences} \\\hline\hline
\multirow{2}{*}{$\phi(L(1,2m)\times D^*_{4n}),\phantom{i}\mathfrak{I}^2_{m,n}$}& 6 & m\not\equiv 1\text{ mod } |\Gamma|/(4m) \\
& 8&m\equiv 1\text{ mod } |\Gamma|/(4m) \\\hline
\multirow{3}{*}{$\phi(L(1,2m)\times T^*)$, $\phi(L(1,2m)\times O^*)$, $\mathfrak{I}^3_m$} &4& m\equiv -1\text{ mod }|\Gamma|/(4m)\\	
						&6& m\not\equiv \pm1\text{ mod }|\Gamma|/(4m)\\	
						&8& m\equiv 1\text{ mod }|\Gamma|/(4m) \\\hline	
\multirow{3}{*}{$\phi(L(1,2m)\times I^*)$} &4& m\equiv 17, 23, 29\text{ mod }30 \\
						& 6& m\equiv 7, 11, 13, 19\text{ mod }30 \\
						&8& m\equiv 1\text{ mod }30 \\\hline												\end{array}
\end{align*}
}
\end{proposition}

\begin{proof}
The space $(\CP^2_{(1,1,2)},g_{BK})$ is the K\"ahler conformal compactification of the Eguchi-Hanson metric on $\mathcal{O}(-2)$ for which it is well known that $\dim(H^1_{SD})=0$.  This is the $m=1$ case.
For $m>1$,  Honda showed that the complexification of $H^1_{SD}$ of the self-dual deformation complex on $(\CP^2_{(1,1,2m)},g_{BK})$ is equivalent to
\begin{align}
\rho\oplus\bar{\rho},\phantom{=}\text{where}\phantom{=}\rho=\big(S^{2m-2}(\CC^2)\otimes \det\big)\oplus \big(S^{2m-4}(\CC^2)\otimes \det{}^2\big),
\end{align}
as a representation space of ${\rm U}(2)$, see \cite{HondaOn}.
The dimension of the space of invariant elements of $H^1_{SD}$ under the action of $\Gamma$ is equal to that under the action of any subgroup $\Gamma'\subset\Gamma$ that has the same effective action as $\Gamma$ and is given by
\begin{align}
\label{character_sum}
\dim(H^1_{SD,\Gamma})=\dim(H^1_{SD,\Gamma'})=\frac{1}{|\Gamma'|}\sum_{\gamma\in\Gamma'}\Big(\chi_{\rho}(\gamma)+\chi_{\bar{\rho}}(\gamma)\Big)=\frac{2}{|\Gamma'|}\sum_{\gamma\in\Gamma'}\chi_{\rho}(\gamma),
\end{align}
where $\chi_{\rho}(\gamma)$ and $\chi_{\bar{\rho}}(\gamma)$ denote the characters of $\gamma$ with respect to each representation.

Since the eigenvalues of any $\gamma\in{\rm U}(2)$ can be written as $\{z_1=e^{i\theta_1},z_2=e^{i\theta_2}\}$ for some $0\leq \theta_1,\theta_2<2\pi$, observe that
\begin{align}
\label{character}
\chi_{\rho}(\gamma)=(z_1z_2)\sum_{p=0}^{2m-2}z_1^{2m-2-p}z_2^p+(z_1z_2)^2\sum_{p=0}^{2m-4}z_1^{2m-4-p}z_2^p.
\end{align}
Clearly, $\dim(H^1_{SD})=2\chi_{\rho}(\pm Id)=8m-8$. 

In order to compute \eqref{character_sum}, we introduce a number theoretic function and identity.
\begin{itemize}
\item For $x\in\RR$, the sawtooth function is defined as
\begin{align}
((x))=\begin{cases}
x-\lfloor x\rfloor -\frac{1}{2}\phantom{=}&\text{when $x\not\in\mathbb{Z}$}\\
0\phantom{=}&\text{when $x\in\mathbb{Z}$},
\end{cases}
\end{align}
where $\lfloor x \rfloor$ denotes the greatest integer less than $x$. 
\item For $n,k$, with $k\geq 1$, the following identity is due to Eisenstein, see \cite{Apostol}:
\begin{align}
\label{Eisenstein_identity}
\sum_{j=1}^{n-1}\sin\Big(\frac{2\pi k}{n}j\Big)\cot\Big(\frac{\pi}{n}j\Big)=-2n\Big(\Big(\frac{k}{n}\Big)\Big).
\end{align}
\end{itemize}

We first consider those groups $\Gamma$ that are the image under $\phi$ of the product of $L(1,2m)$ and a binary polyhedral group, and
let $\Gamma'\subset\Gamma$ be the subgroup that is the binary polyhedral group itself.  All elements 
$\gamma\in \Gamma'$ have eigenvalues of the form $\{z=e^{i\theta_{\gamma}},z^{-1}=e^{-i\theta_{\gamma}}\}$ since each $\Gamma'\subset{\rm SU}(2)$.  Thus,  for $\gamma\neq \pm 1$, the character \eqref{character} reduces to
\begin{align}
\begin{split}
\label{character_sum_II}
\chi_{\rho}(\gamma)&=\sum_{p=0}^{2m-2}z^{2m-2-2p}+\sum_{p=0}^{2m-4}z^{2m-4-2p}=2\sin((2m-2)\theta_{\gamma})\cot(\theta_{\gamma}),
\end{split}
\end{align}
and therefore
\begin{align}
\label{invariant_elements_SU(2)}
\dim(H^1_{SD,\Gamma})=\dim(H^1_{SD,\Gamma'})=\frac{4}{|\Gamma'|}\Big[(4m-4)+\sum_{\gamma\neq \pm Id\in \Gamma'}\sin((2m-2)\theta_{\gamma})\cot(\theta_{\gamma})\Big].
\end{align}
Now, $\dim(H^1_{SD,\Gamma'})$ will be found for each binary polyhedral group separately.  We provide the eigenvalue decomposition of each group in the form of a list of sets of eigenvalues along with multiplicities, where the multiplicity of a particular set is the number of times the eigenvalues of that set appear in the set of all eigenvalues of the group.  Grouping the eigenvalues as such will allow us to use \eqref{Eisenstein_identity} to compute \eqref{invariant_elements_SU(2)}.  The eigenvalue decompositions of the binary polyhedral groups are found easily by examining their well-known decomposition into conjugacy classes, see \cite{Stekolshchik}.
\\

\noindent
{$\boldsymbol {\Gamma'=D^{\ast}_{4n}:}$} 
The eigenvalue decomposition of the binary dihedral group is given by:
\begin{align*}
\begin{array}{l|c}
\hline
\text{Set}&\text{Multiplicity}\\\hline\hline
S_1=\Big\{ \{e^{\frac{\pi i k}{n}},e^{-\frac{\pi i k}{n}}\}\Big\}_{1\leq k\leq 2n}& 1\\
S_2=\Big\{\{i,-i\}\Big\}& 2n
\end{array}
\end{align*}
Summing the characters \eqref{character_sum_II} of the elements of $S_1$ that are not $\pm Id$ gives two copies of the sum \eqref{Eisenstein_identity} where $k=m-1$.
The element of $S_2$ does not contribute to the sum in \eqref{invariant_elements_SU(2)} since $((1/2))=0$.    Therefore
\begin{align*}
\begin{split}
\dim(H^1_{SD,D^*_{4n}})&=\frac{1}{n}\Big[(4m-4)-4n\Big(\Big(\frac{m-1}{n}\Big)\Big)\Big]\\
&=\begin{cases}
4\lfloor\frac{m-1}{n}\rfloor+2=4b_{\phi(L(1,2m)\times D^*_{4n})}-6\phantom{=}&m\not\equiv 1\text{ mod }n\\
\frac{4(m-1)}{n}=4b_{\phi(L(1,2m)\times D^*_{4n})}-8\phantom{=}&m\equiv 1\text{ mod }n.
\end{cases}
\end{split}
\end{align*}

\noindent
{$\boldsymbol {\Gamma'=T^*:}$}
The eigenvalue decomposition of the binary tetrahedral group is given by:
\begin{align*}
\begin{array}{l|c}
\hline
\text{Set}&\text{Multiplicity}\\\hline\hline
S_1=\Big\{ \{1,1\}, \{-1,-1\}\Big\}&1\\
S_2=\Big\{ \{e^{\frac{\pi i}{3}},e^{-\frac{\pi i}{3}}\}, \{e^{\frac{2\pi i}{3}},e^{-\frac{2\pi i}{3}}\}\Big\}&8\\
S_3=\Big\{ \{i,-i\}\Big\}&6
\end{array}
\end{align*}
Use \eqref{Eisenstein_identity} to sum the characters of the elements of $S_2$ as given by \eqref{character_sum_II}.  The elements of $S_3$ do not contribute to the sum in \eqref{invariant_elements_SU(2)} since $((1/2))=0$.  Therefore, adjusting these sums according to the appropriate multiplicities, we find that
\begin{align*}
\begin{split}
\dim(H^1_{SD,T^*})&=\frac{1}{6}\Big[(4m-4)-48\Big(\Big(\frac{m-1}{3}\Big)\Big)\Big]\\
&=\begin{cases}
\frac{4(m+1)}{6}=4b_{\phi(L(1,2m)\times T^*)}-4\phantom{=}&m\equiv 5\text{ mod }6\\
\frac{4(m-1)}{6}=4b_{\phi(L(1,2m)\times T^*)}-8\phantom{=}&m\equiv 1\text{ mod }6.
\end{cases}
\end{split}
\end{align*}

\noindent
{$\boldsymbol {\Gamma'=O^*:}$} 
The eigenvalue decomposition of the binary octahedral group is given by:
\begin{align*}
\begin{array}{l|c}
\hline
\text{Set}&\text{Multiplicity}\\\hline\hline
S_1=\Big\{ \{1,1\}, \{-1,-1\}\Big\}&1\\
S_2=\Big\{ \{e^{\frac{\pi i}{3}},e^{-\frac{\pi i}{3}}\}, \{e^{\frac{2\pi i}{3}},e^{-\frac{2\pi i}{3}}\}\Big\}&8\\
S_3=\Big\{ \{e^{\frac{\pi i}{4}},e^{-\frac{\pi i}{4}}\}, \{i,-i\}, \{e^{\frac{3\pi i}{4}},e^{-\frac{3\pi i}{4}}\}\Big\}& 6\\
S_3=\Big\{ \{i,-i\}\Big\}&12
\end{array}
\end{align*}
Use \eqref{Eisenstein_identity} to sum the characters of the elements of $S_2$ and $S_3$ as given by \eqref{character_sum_II}.  The elements of $S_4$ do not contribute to the sum in \eqref{invariant_elements_SU(2)} since $((1/2))=0$.  Therefore, adjusting these sums according to the appropriate multiplicities, we find that
 \begin{align*}
 \begin{split}
 \dim(H^1_{SD,O^*})&=\frac{1}{12}\Big[(4m-4)-48\Big(\Big(\frac{m-1}{3}\Big)\Big)-48\Big(\Big(\frac{m-1}{4}\Big)\Big)\Big]\\
&=\begin{cases}
\frac{m+1}{3}=b_{\phi(L(1,2m)\times O^*)}-4\phantom{=}&m\equiv 11\text{ mod }12\\
\frac{m-1}{3}=b_{\phi(L(1,2m)\times O^*)}-6\phantom{=}&m\equiv 7\text{ mod }12\\
\frac{m+1}{3}=b_{\phi(L(1,2m)\times O^*)}-6\phantom{=}&m\equiv 5\text{ mod }12\\
\frac{m-1}{3}=b_{\phi(L(1,2m)\times O^*)}-8\phantom{=}&m\equiv 1\text{ mod }12.
\end{cases}
\end{split}
\end{align*}

\noindent
{$\boldsymbol {\Gamma'=I^*:}$} 
The eigenvalue decomposition of the binary octahedral group is given by:
\begin{align*}
\begin{array}{l|c}
\hline
\text{Set}&\text{Multiplicity}\\\hline\hline
S_1=\Big\{ \{1,1\}, \{-1,-1\}\Big\}&1\\
S_2=\Big\{ \{e^{\frac{\pi i}{3}},e^{-\frac{\pi i}{3}}\}, \{e^{\frac{2\pi i}{3}},e^{-\frac{2\pi i}{3}}\}\Big\}&20\\
S_3=\Big\{ \{e^{\frac{\pi i}{5}},e^{-\frac{\pi i}{5}}\}, \{e^{\frac{2\pi i}{5}},e^{-\frac{2\pi i}{5}}\},\{e^{\frac{3\pi i}{5}},e^{-\frac{3\pi i}{5}}\}, \{e^{\frac{4\pi i}{5}},e^{-\frac{4\pi i}{5}}\}\Big\}& 12\\
S_3=\Big\{ \{i,-i\}\Big\}&30
\end{array}
\end{align*}
Use \eqref{Eisenstein_identity} to sum the characters of the elements of $S_2$ and $S_3$ as given by \eqref{character_sum_II}.  The elements of $S_4$ do not contribute to the sum in \eqref{invariant_elements_SU(2)} since $((1/2))=0$.  Therefore, adjusting these sums according to the appropriate multiplicities, we find that
\begin{align*}
\begin{split}
\dim(H^1_{SD,I^*})&=\frac{1}{30}\Big[(4m-4)-120\Big(\Big(\frac{m-1}{3}\Big)\Big)-120\Big(\Big(\frac{m-1}{5}\Big)\Big)\Big]\\
&=\begin{cases}
\frac{2(m+1)}{15}=4b_{\phi(L(1,2m)\times I^*)}-4\phantom{=}&m\equiv 29\text{ mod }30\\
\frac{2(m+7)}{15}=4b_{\phi(L(1,2m)\times I^*)}-4\phantom{=}&m\equiv 23\text{ mod }30\\
\frac{2(m+13)}{15}=4b_{\phi(L(1,2m)\times I^*)}-4\phantom{=}&m\equiv 17\text{ mod }30\\
\frac{2(m-4)}{15}=4b_{\phi(L(1,2m)\times I^*)}-6\phantom{=}&m\equiv 19\text{ mod }30\\
\frac{2(m+2)}{15}=4b_{\phi(L(1,2m)\times I^*)}-6\phantom{=}&m\equiv 13\text{ mod }30\\
\frac{2(m+8)}{15}=4b_{\phi(L(1,2m)\times I^*)}-6\phantom{=}&m\equiv 7\text{ mod }30\\
\frac{2(m+4)}{15}=4b_{\phi(L(1,2m)\times I^*)}-6\phantom{=}&m\equiv 11\text{ mod }30\\
\frac{2(m-1)}{15}=4b_{\phi(L(1,2m)\times I^*)}-8\phantom{=}&m\equiv 1\text{ mod }30.
\end{cases}
\end{split}
\end{align*}

For the $\mathfrak{I}^2_{m,n}$ and $\mathfrak{I}^3_m$ cases, we will compute \eqref{character_sum} over the entire group $\Gamma$ since there is not a clear subgroup to play the role of $\Gamma'$ as above.  
Since these groups are not contained in ${\rm SU}(2)$, 
not all $\gamma\in\Gamma$ have eigenvalues of the form $\{e^{i\theta_{\gamma}},e^{-i\theta_{\gamma}}\}$, 
and therefore formula $\eqref{invariant_elements_SU(2)}$ does not hold.  In \cite[Proposition $6.1$]{LVsmorgasbord}, the authors performed an eigenvalue decomposition of these groups, which will be used to find $\dim(H^1_{SD,\Gamma})$, and from which we know that all elements of both groups have eigenvalues of the form $\{e^{i(\theta_1+\theta_2)},e^{i(\theta_1-\theta_2)}\}$.  Therefore, we compute
\begin{align}
\begin{split}
\label{character_sum_III}
\chi_{\rho}(\gamma)
&=e^{i(2m\theta_1)}\Big[\sum_{p=0}^{2m-2}\big(e^{i\theta_2}\big)^{2m-2-2p}+\sum_{p=0}^{2m-4}\big(e^{i\theta_2}\big)^{2m-4-2p}\Big]\\
&=\begin{cases}
\displaystyle e^{i(2m\theta_1)}(4m-4)\phantom{=}&\theta_2=\ell\pi\text{ for }\ell\in\mathbb{Z}\\
\displaystyle e^{i(2m\theta_1)}\big[2\sin((2m-2)\theta_2)\cot(\theta_2)\big]\phantom{=}&\text{$\theta_2\neq\ell\pi$, for $\ell\in\mathbb{Z}$}.
\end{cases}
\end{split}
\end{align}

\noindent
{$ \boldsymbol{\Gamma=\mathfrak{I}^2_{m,n}:}$} 
The eigenvalue decomposition of $\mathfrak{I}^2_{m,n}$ is given by:
\begin{align*}
\begin{array}{l|c}
\hline
\text{Set}&\text{Multiplicity}\\\hline\hline
S_1=\Big\{(-1)^\ell\{e^{\pi i (\frac{\ell}{m}+\frac{k}{n})},e^{\pi i (\frac{\ell}{m}-\frac{k}{n})}\}\Big\}_{0\leq \ell\leq m-1 \text{ and }0\leq k\leq 2n-1}& 1\\
S_2=\Big\{(-1)^\ell\{e^{\pi i (\frac{2\ell+1}{2m}+\frac{1}{2})},e^{\pi i (\frac{2\ell+1}{2m}-\frac{1}{2})}\}\Big\}_{0\leq \ell\leq m-1 \text{ and }0\leq k\leq 2n-1}& 1
\end{array}
\end{align*}
Using \eqref{character_sum_III}, the characters are found to be
\begin{align*}
\chi_{\rho}(\gamma\in S_1)=&\begin{cases}
2\sin\big(\frac{2\pi}{n}k(m-1)\big)\cot\big(\frac{\pi}{n}k\big)\phantom{=}&\text{$k\neq 0$ or $n$}\\
4m-4\phantom{=}&\text{$k=0$ or $n$}.
\end{cases}\\
\chi_{\rho}(\gamma\in S_2)=&0\phantom{=}\text{ since $\theta_2=\pi/2$.}\\
\end{align*}
Then, using \eqref{Eisenstein_identity} to sum the characters, we find that
\begin{align*}
\begin{split}
\dim(H^1_{SD,\mathfrak{I}^2_{m,n}})&=\frac{1}{2mn}\Big[2m\cdot(4m-4)-8mn\Big(\Big(\frac{m-1}{n}\Big)\Big)\Big]\\
&=\begin{cases}
4\lfloor\frac{m-1}{n}\rfloor+2=4b_{\mathfrak{I}^2_{m,n}}-6\phantom{=}&m\not\equiv 1\text{ mod }n\\
\frac{4(m-1)}{n}=4b_{\mathfrak{I}^2_{m,n}}-8\phantom{=}&m\equiv 1\text{ mod }n.
\end{cases}\end{split}
\end{align*}

\noindent
{$ \boldsymbol{\Gamma=\mathfrak{I}^3_m:}$} 
The eigenvalue decomposition of $\mathfrak{I}^3_m$ is given by:
\begin{align*}
\begin{array}{l|c}
\hline
\text{Set}&\text{Multiplicity}\\\hline\hline
S_1=\Big\{\pm e^{\frac{\pi ir}{m}}\{1,1\} \Big\}_{0\leq r<m} & 1\\
S_2=\Big\{\pm e^{\frac{\pi ir}{m}}\{i,-i\} \Big\}_{0\leq r<m}& 3\\
S_3=\Big\{\pm e^{\frac{\pi i(3r+1)}{3m}}\{e^{\frac{\pi i}{3}},e^{\frac{-\pi i}{3}}\}\Big\}_{0\leq r<m}  & 2\\
S_4=\Big\{\pm e^{\frac{\pi i(3r+1)}{3m}}\{e^{\frac{2\pi i}{3}},e^{\frac{-2\pi i}{3}}\}\Big\}_{0\leq r<m} &2\\
S_5=\Big\{\pm e^{\frac{\pi i(3r+2)}{3m}}\{e^{\frac{2\pi i}{3}},e^{\frac{-2\pi i}{3}}\}\Big\}_{0\leq r<m} &4
\end{array}
\end{align*}
Using \eqref{character_sum_III}, the characters are found to be
\begin{align}
\begin{split}
\label{index_3_characters}
\chi_{\rho}(\gamma\in S_1)=&4m-4\\
\chi_{\rho}(\gamma\in S_2)=&0\\
\chi_{\rho}(\gamma\in S_3)=&2e^{2\pi i/3}\sin\Big(2(m-1)\frac{\pi}{3}\Big)\cot\Big(\frac{\pi}{3}\Big)=-e^{2\pi i/3}\\
\chi_{\rho}(\gamma\in S_4)=&2e^{2\pi i/3}\sin\Big(2(m-1)\frac{2\pi}{3}\Big)\cot\Big(\frac{2\pi}{3}\Big)=-e^{2\pi i/3}\\
\chi_{\rho}(\gamma\in S_5)=&2e^{4\pi i/3}\sin\Big(2(m-1)\frac{2\pi}{3}\Big)\cot\Big(\frac{2\pi}{3}\Big)=-e^{4\pi i/3}.
\end{split}
\end{align}
The evaluation of the last three characters follows from $m\equiv 0\text{ mod }3$.
Since all eigenvalues contained in the same set have equal characters, by multiplying each character found in \eqref{index_3_characters} by $2m$, which is the size of each set, and the multiplicity of the respect set, and summing, we find that
\begin{align*}
\dim(H^1_{SD,\mathfrak{I}^3_m})&=\frac{1}{12m}\Big[2m\cdot(4m-4)-8m(e^{2\pi i/3}+e^{4\pi i/3})\Big]=\frac{2}{3}m\\
&=4b_{\mathfrak{I}^3_m}-6.
\end{align*}

\end{proof}

Now, we complete the proof of Theorem \ref{index_theorem}, by finding $N(\Gamma)$, the non-topological correction term for the index, for all finite non-cyclic $\Gamma\subset{\rm U}(2)$.
The cohomology groups of the self-dual deformation complex on the quotient $(\CP^2_{(1,1,2m)},g_{BK})/\Gamma$ correspond to the invariant elements of the respective cohomology groups on the cover, which we denote by $H^i_{SD, \Gamma}$ for $i=1,2,3$.
Given $\dim(H^1_{SD, \Gamma})$, found in Proposition \ref{invariant_proposition}, it is only necessary to find $\dim(H^0_{SD,\Gamma})$ and $\dim(H^2_{SD,\Gamma})$ to recover the index.

The second cohomology group $H^2_{SD}$ of the self-dual deformation complex 
for the Bochner-K\"ahler metric on the
weighted projective space $(\CP^2_{(1,1,2m)},g_{BK})$ vanishes by \cite[Theorem 4.2]{LeBrunMaskit}, and therefore in the quotient it vanishes as well.  The cohomology group $H^0_{SD,\Gamma}$ is given by the elements of $H^0_{SD}$ on the cover that commute with the respective $\Gamma$.
The $S^1$ given by $\phi(e^{i\theta},1)$ is contained in the centralizer of all $\Gamma$, and it is easy to check that, for all non-cyclic $\Gamma$, these are the only elements of $H^0_{SD}$ which commute with the respective $\Gamma$, so $\dim(H^0_{SD,\Gamma})=1$.
Therefore, the index is 
 \begin{align}
 \begin{split}
 \label{index_1}
Ind((\CP^2_{(1,1,2m)},g_{BK})/\Gamma)=1-\dim(H^1_{SD,\Gamma}).
\end{split}
\end{align}
 
The quotient $(\CP^2_{(1,1,2m)},g_{BK})/\Gamma$ has four singularities with orbifold groups $\Gamma$ and $L(q_i,p_i)$, for $i=1,2,3$, as specified above in Theorem \ref{quotient_theorem}.  Thus, from \eqref{index_quantities}, the index is also given by
\begin{align}
\label{index_2}
Ind((\CP^2_{(1,1,2m)},g_{BK})/\Gamma)=\frac{1}{2}(15\chi_{top}-29\tau_{top})+\sum_{i=1}^3N(L(q_i,p_i))+N(\Gamma).
\end{align}

Equating \eqref{index_1} and \eqref{index_2}, since 
$\chi_{top}=3$ and $\tau_{top}=1$, we find that
\begin{align}
\label{correction_formula}
N(\Gamma)=-7-\dim(H^1_{SD,\Gamma})-\sum_{i=1}^3N(L(q_i,p_i)).
\end{align} 
For each $\Gamma$, insert the corresponding $\dim(H^1_{SD,\Gamma})$ and cyclic orbifold groups, as were found in Proposition \ref{invariant_proposition} and Theorem \ref{quotient_theorem} respectively.  Finally, the proof is completed, for $\Gamma\subset{\rm U}(2)$, by examining the conditions on the $\Gamma$ and the respective possible congruences of $m$, and using these appropriately along with the known cyclic correction terms \eqref{N(L(q,p))} to see that
\begin{align}
N(\Gamma)=-4b_{\Gamma}+\mathcal{B}_{\Gamma},
\end{align}
where the $\mathcal{B}_{\Gamma}$ are constants given in Table \ref{index_correction}.  

We compute $N(\phi(L(1,2m)\times T^*))$ as an example.  From Theorem \ref{quotient_theorem} we 
know that the singularities are of type $L(1,2),L(m,3),L(m,3)$, so using \eqref{N(L(q,p))} to find the correction terms for these singularities along with $H^1_{SD,\Gamma}$ as determined in
Proposition~\ref{invariant_proposition}, observe that
\begin{align*}
N(\phi(L(1,2m)\times T^*))=&-7-N(L(1,2))-2N(L(m,3))\\
&-4b_{\phi(L(1,2m)\times T^*)}+\begin{cases}
8\phantom{=}&m\equiv 1\text{ mod } 6\\
6\phantom{=}&m\equiv 5\text{ mod } 6
\end{cases}\\
=&-4b_{\phi(L(1,2m)\times T^*)}+ \begin{cases}
21\phantom{=}&m\equiv 1\text{ mod } 6\\
5\phantom{=}&m\equiv 5\text{ mod } 6.
\end{cases}
\end{align*}
The idea for the other cases is identical and the computations, which follow similarly, are omitted.

\begin{table}[ht]
\caption{Self-dual index correction terms for $\Gamma\subset{\rm U}(2)$}
\begin{center}
\label{index_correction}
\renewcommand\arraystretch{1.4}
\begin{tabular}{|l|c|l|}
\hline
$\Gamma\subset \rm{U}(2)$& $\mathcal{B}_{\Gamma}$ & Congruences  \\\hline\hline

\multirow{2}{*}{$\phi(L(1,2m)\times D^*_{4n})$} & $7-N(L(m,n))$&$m\not\equiv 1$ mod $n$ \\
& $5+4n$ &$m\equiv 1$ mod $n$      \\\hline

\multirow{2}{*}{$\phi(L(1,2m)\times T^*)$}      & $5$    &  $m\equiv 5$ mod $6$       \\
							& $21$& $m\equiv 1$ mod $6$		\\\hline							
\multirow{4}{*}{$\phi(L(1,2m)\times O^*)$} & $1$ & $m\equiv 11$ mod $12$ \\
					 & $9$&  $m\equiv 7$ mod $12$\\
					& $17$& $m\equiv 5$ mod $12$  \\
					& $25$& $m\equiv 1$ mod $12$ \\\hline			
\multirow{6}{*}{$\phi(L(1,2m)\times I^*)$}& $-3$ & $m\equiv 29$ mod $30$ \\
					& $5$& $m\equiv 19$ mod $30$ \\	
					& $9$& $m\equiv 17,23$ mod $30$\\
					& $17$& $m\equiv 7,13$  mod $30$ \\
					&  $21$& $m\equiv 11$ mod $30$	\\
					& $29$& $m\equiv 1$ mod $30$ \\\hline
\multirow{2}{*}{$\mathfrak{I}^2_{m,n}$ diagonal} & $7-N(L(m,n))$ &$m\not\equiv 1$ mod $n$ \\
	& $5+4n$& $m\equiv 1$ mod $n$        \\\hline
$\mathfrak{I}^3_m$ diagonal & $13$& $m\equiv 3$ mod $6$	\\\hline	
\end{tabular}
\end{center}
\end{table}

\begin{remark}
{\em Although the correction terms for $\phi(L(1,2m)\times D^*_{4n})$ and $\mathfrak{I}^2_{m,n}$ contain an $N(L(m,n))$, they are computed algorithmically via \eqref{N(L(q,p))}. 
}
\begin{remark}
\label{SD_index_m=1}
{\em
The second author found the correction terms for the binary polyhedral groups \cite{ViaclovskyIndex}.  These are recovered here as well in the $m=1$ case.}
\end{remark}

\end{remark}

\subsection{Proof of Theorem \ref{more_ASD_deformations}}
\label{pf_masdd}
Let $(X,g)$ be a scalar-flat K\"ahler ALE metric on the minimal resolution of $\CC^2/\Gamma$ for some finite subgroup $\Gamma\subset{\rm U}(2)$ which acts freely on $S^3$.
In \cite{LVsmorgasbord}, the authors showed that the dimension of infinitesimal scalar-flat K\"ahler deformations is at most
\begin{align}
d_{\Gamma, max} = 2  \Big( \sum_{i = 1}^{k_{\Gamma}} (e_i -1) \Big)  + k_{\Gamma} - 3,
\end{align}
where $-e_i$ is the self-intersection number of the $k$th rational curve;  hence, to prove
Theorem~\ref{more_ASD_deformations}, we will consider the self-dual conformal compactification $(\widehat{X},\widehat{g})$, and show that dimension of the moduli space of self-dual conformal structures near $\widehat{g}$ is greater than or equal to $d_{\Gamma, max}$, with equality if and only if $\Gamma\subset{\rm SU}(2)$.  We separate
the proof into two parts -- for $\Gamma$ non-cyclic and for $\Gamma$ cyclic.  The underlying idea is the same for each case, but due to the differences of the respective $N(\Gamma)$ it is necessary to consider them separately.

Let $\Gamma$ be non-cyclic.  To find a convenient presentation of $d_{\Gamma, max}$, consider the 
description of the minimal resolution of $\CC^2/\Gamma$ from Section \ref{minimal_resolutions}.  Recall there is the central rational curve with self-intersection $-b_{\Gamma}$ off of which three Hirzebruch-Jung strings corresponding to the singularities $L(p_1-q_i,p_i)$, for $i=1,2,3$, emanate.
Letting $E_{(p_i-q_i,p_i)}$ denote the sum $\sum_{j=1}^{k_{(p_i-q_i,p_i)}}e^{(p_i-q_i,p_i)}_j$, where the $e^{(p_i-q_i,p_i)}_j$ are the coefficients appearing in the modified Euclidean algorithm for the pair $(p_i-q_i,p_i)$, we see that
 \begin{align}
 d_{\Gamma, max}=\sum_{i=1}^3(2E_{(p_i-q_i,p_i)}-k_{(p_i-q_i,p_i)})+2b_{\Gamma}-4.
 \end{align}
Now, we examine the index of the self-dual deformation complex on $(\widehat{X},\widehat{g})$.
Recall that $\tau_{top}(\widehat{X})=1+\sum_{i=1}^3k_{(p_i-q_i,p_i)}$ and $\chi_{top}(\widehat{X})=\tau_{top}(\widehat{X})+2$.  Therefore, since  $H^2_{SD}(\widehat{X}) $ of the self-dual deformation complex vanishes \cite[Theorem 4.2]{LeBrunMaskit}, the index is given by
\begin{align}
\label{NGamma_index}
Ind(\widehat{X},\widehat{g})=\dim(H^0_{SD})-\dim(H^1_{SD})=-7\sum_{i=1}^3k_{(p_i-q_i,p_i)}+8+N(\Gamma).
\end{align}
Using \eqref{correction_formula} for $N(\Gamma)$, Proposition \ref{invariant_proposition} for $H^1_{SD,\Gamma}$, and since $\dim(H^0_{SD})=1$, observe that
\begin{align}
\label{H1_+}
\dim(H^1_{SD})=7\sum_{i=1}^3k_{(p_i-q_i,p_i)}+4b_{\Gamma}-\mathcal{C}_{\Gamma}+\sum_{i=1}^3N(L(q_i,p_i)),
\end{align}
where the $\mathcal{C}_{\Gamma}$ are constants specified in Proposition \ref{invariant_proposition}.
Although $N(\Gamma)$ is given explicitly in Table \ref{index_correction}, it is more
useful to consider \eqref{correction_formula} here.
We would like to understand each $N(L(q_i,p_i))$ in terms of $N(L(p_i-q_i,p_i))$ as to be able to compare \eqref{H1_+} with $d_{\Gamma, max}$.  
In \cite{LockViaclovsky}, the authors showed that, for $1<q_i<p_i-1$,
\begin{align}
N(L(q_i,p_i))=-N(L(p_i-q_i,p_i))-12=4E_{(p_i-q_i,p_i)}-12k_{(p_i-q_i,p_i)}-2,
\end{align}
and from \eqref{index_quantities} it is clear that $N(L(1,p_i))=-N(L(-1,p_i))-10$ for $p_i~>~2$.  Therefore, we find that
\begin{align}
\dim(H^1_{SD})=\sum_{i=1}^3\big(4E_{(p_i-q_i,p_i)}-5k_{(p_i-q_i,p_i)}\big)+4b_{\Gamma}-\mathcal{C}_{\Gamma}-2\kappa,
\end{align}
where $\kappa$ is the number of singularities not of type $L(\pm 1,p_i)$; note that $k\leq 1$ with equality only in the case that $\Gamma=\phi(L(1,2m)\times I^*)$ and $m\equiv \pm 2\text{ mod }5$.  Using \eqref{b_Gamma}, observe that
\begin{align}
\begin{split}
\label{H-d}
\dim(H^1_{SD})-d_{\Gamma, max}&=\sum_{i=1}^3\big(2E_{(p_i-q_i,p_i)}-4k_{(p_i-q_i,p_i)}\big)\\
&\ \ \ +\frac{8m}{|\Gamma|}\Big[m-\Big(m\text{ mod }\frac{|\Gamma|}{4m}\Big)\Big]-\mathcal{C}_{\Gamma}-2\kappa+8.
\end{split}
\end{align}
For $\Gamma\subset {\rm SU}(2)$, it is clear from \eqref{H-d} that $\dim(H^1_{SD})=d_{\Gamma, max}$.  These are the hyperk\"ahler metrics, see \cite[Section $8$]{LVsmorgasbord}.  For $\Gamma \not\subset  {\rm SU}(2)$, each $E_{(p_i-q_i,p_i)}>2k_{(p_i-q_i,p_i)}$ since at least one of the $e^{(p_i-q_i,p_i)}_j>2$.  Also, $\mathcal{C}_{\Gamma}\leq 8$, and in particular $\mathcal{C}_{\Gamma}\leq 6$ for the $\Gamma$ where $\kappa=1$.  Therefore, we see that
\begin{align}
\dim(H^1_{SD})-d_{\Gamma, max}\geq\sum_{i=1}^3\big(2E_{(p_i-q_i,p_i)}-4k_{(p_i-q_i,p_i)}\big)>0.
\end{align}

Now, let $\Gamma=L(q,p)$, for some relatively prime integers $1\leq q<p$, so
\begin{align}
d_{\Gamma, max}=\sum_{j=1}^{k_{(q,p)}}2e_j-k_{(q,p)}-3,
\end{align}
where the $e_j$ are as in the modified Euclidean algorithm for the pair $(q,p)$.  Now, we examine the index of the self-dual deformation complex on $(\widehat{X},\widehat{g})$.  Recall that $\tau_{top}(\widehat{X})=k_{(q,p)}$ and $\chi_{top}(\widehat{X})=k_{(q,p)}+2$.  Once again, from \cite[Theorem 4.2]{LeBrunMaskit}, we see that $H^2_{SD}$ vanishes.  Therefore
\begin{align}
Ind(\widehat{X},\widehat{g})=\dim(H^0_{SD})-\dim(H^1_{SD})=-7k_{(q,p)}+15+N(L(q,p)),
\end{align}
and using \eqref{N(L(q,p))} we find that
\begin{align}
\dim(H^1_{SD})=\dim(H^0_{SD})+\sum_{j=1}^{k_{(q,p)}}4e_j-5k_{(q,p)}-5-2\kappa,
\end{align}
where $\kappa=1$ if $q=1$ and $\kappa=0$ otherwise.  Observe that
\begin{align}
\dim(H^1_{SD})-d_{\Gamma, max}=\dim(H^0_{SD})+\sum_{j=1}^{k_{(q,p)}}2e_j-4k_{(q,p)}-2-2\kappa.
\end{align}
For $\Gamma=L(-1,p)\subset{\rm SU}(2)$, these are the hyperk\"ahler mutli-Eguchi-Hanson metrics, and it is well known that the moduli space is of dimension $3(p-2)=d_{\Gamma, max}$.  For $\Gamma\not\subset {\rm SU}(2)$, at least one $e_j>2$ and thus $\dim(H^1_{SD})-d_{\Gamma, max}>0$ which completes the proof.


\section{Self-dual constructions}
\label{new_spaces}
It is well known that self-dual orbifolds with complementary singularities, and both with vanishing $H^2_{SD}$ of the self-dual deformation complex, can be glued together to obtain self-dual metrics on the connected sum, see \cite{DonaldsonFriedman,Floer,LeBrunSinger,KovalevSinger,AcheViaclovsky2}.  The following theorem summarizes the results of these works.
\begin{theorem}
\label{SD_gluing}
Let $(M_1,[g_1])$ and $(M_2,[g_2])$ be compact self-dual orbifolds which have complementary singularities, i.e. the respective orbifold groups are orientation reversed conjugate to each other.  If the second cohomology group of the self-dual deformation complex on each orbifold vanishes, then the connect sum $M_1\#M_2$, taken at the complementary singularities, admits self-dual metrics.
\end{theorem}

Recall the list of the finite subgroups of ${\rm SO}(4)$ which act freely on $S^3$, preserve the Hopf fibration and which are not contained in ${\rm U}(2)$ that was given in Table \ref{orientation_reversed_groups}.
Since these groups do not lie in ${\rm U}(2)$, a scalar-flat K\"ahler ALE space with such a group at infinity cannot exist.
Thus, as posed by the second author in \cite{ViaclovskyIndex}, the natural question is that of the existence of scalar-flat anti-self-dual ALE metrics with these groups at infinity. Theorem~\ref{tor} answers this question in the affirmative for these groups, which is proved below in Section \ref{proof_reversedALE_thm}.  Subsequently, Corollary \ref{tor_corollary} is proved in Section~\ref{sequence_corollary_proof}.

\subsection{Proof of Theorem \ref{tor}}
\label{proof_reversedALE_thm}
Let $\overline{\Gamma_1}=\overline{\phi(L(1,2m)\times D^*_{4n})}$ and $\overline{\Gamma_2}=\overline{\mathfrak{I}^2_{m,n}}$.  Since both preserve the Hopf fiber structure,
they act isometrically on the LeBrun negative mass metrics.  First, for each of these groups, we will take the quotient of some appropriate negative mass metric to obtain an orbifold ALE space with the action of the entire group at infinity.  The idea here is analogous to that of \cite[Theorem 4.1]{LVsmorgasbord}. 
Notice that both $\overline{\Gamma_1}$ and $\overline{\Gamma_2}$ contain the generator $\phi(e^{\pi i/n},1)$, and as such, respectively for each we take the quotient of $(\mathcal{O}(-2n),g_{LB})$ by the subgroup 
\begin{align}
\overline{\Gamma_i}'=\begin{cases}
\langle \phi(1,e^{\pi i/m}), \phi(\hat{j},1)\rangle\phantom{=}&i=1\\
\langle \phi(\hat{j},e^{\pi i/(2m)})\rangle\phantom{=}&i=2.
\end{cases}
\end{align}
Here, the action is induced by the usual action on $\CC^2$ and, in particular, on the $\CP^1$ at the origin it descends via the Hopf map.
Since $\overline{\Gamma_i}'\not\subset{\rm U}(2)$ the K\"ahlerian property is not preserved in the quotient,
however, the anti-self-dual property is preserved.

Since $\hat{j}*(z_1+z_2\hat{j})=\bar{z}_1\hat{j}-\bar{z}_2$, observe that under the Hopf map $\phi(\hat{j},1)$ acts as
\begin{align}
\mathcal{H}\big(\phi(\hat{j},1)(z_1,z_2)\big)= -\frac{\bar{z}_2}{\bar{z}_1}\in S^2=\CC\cup\{\infty\},
\end{align}
which is the antipodal map.  Also, notice $\phi(\hat{j},e^{\pi i/(2m)})^2=\phi(-1,e^{\pi i/m})$ fixes points on the $\CP^1$ at the origin that are identified by the antipodal map (the points $\{0\},\{\infty\}\in S^2=\CC\cup\{\infty\}$).
Therefore, similar to the work in \cite[Theorem 4.1]{LVsmorgasbord}, 
we find that the quotients $(\mathcal{O}(-2n),g_{LB})/\overline{\Gamma_i}'$, for $i=1,2$, are anti-self dual ALE orbifolds with  groups at infinity the respective $\overline{\Gamma_i}$, and each having one singularity with orbifold group $L(-n,m)$ on the $\RP^2$ resulting from the quotient of the $\CP^1$ at the origin by the antipodal map.  Clearly,  this space has $\pi_1(\mathcal{O}(-2n)/\overline{\Gamma_i}')=\mathbb{Z}/2\mathbb{Z}$ and $\tau_{top}(\mathcal{O}(-2n)/\overline{\Gamma_i}')=0$.  Notice that when $m=1$, which can only occur for $\Gamma_2$, these are in fact smooth quotients and the proof is complete.  Therefore, from here on we can assume $m> 1$.

We note that $O(-2n) /\overline{\Gamma_i}'$ is diffeomorphic to a non-orientable bundle over $\RP^2$. These are classified by $H^2( \RP^2, \mathbb{Z}_{w})$, where $Z_w$ is the system of local coefficients determined by the first Stiefel-Whitney class $w$ of the bundle. It turns out that $H^2( \RP^2, \mathbb{Z}_w)$ is isomorphic to $Z$, and under a suitable choice of isomorphism, this bundle is mapped to the integer $-n$, see \cite{Baraglia}.

The remainder of the proof will follow from the self-dual orbifold gluing of Theorem~\ref{SD_gluing}.  
Consider the self-dual conformal compactification $(\widehat{\mathcal{O}}(-2n),\widehat{g}_{LB})/\overline{\Gamma_i}'$ which has two singularities -- one with orbifold group $L(n,m)$ on the $\RP^2$ and the other with orbifold group $\overline{\Gamma_i}$ at the point of compactification, for $i=1,2$.  Let
$(\widehat{C}_{L(-n,m)},\widehat{h})$ denote the compactification of a Calderbank-Singer metric with group $L(-n,m)$ at infinity.
The second cohomology group $H^2_{SD}$ of the self-dual deformation complex of both orbifolds vanishes here by \cite[Theorem 4.2]{LeBrunMaskit} respectively.  Therefore, we can apply
Theorem \ref{SD_gluing} to obtain a self-dual orbifold conformal structure $[\widehat{g}_i]$ with positive Yamabe invariant on
\begin{align}
\widehat{X}_i=\widehat{\mathcal{O}}(-2n)/\overline{\Gamma_i}'\phantom{i}\#\phantom{i} \widehat{C}_{L(-n,m)},
\end{align}
where the connect sum is taken at the $L(n,m)$ orbifold point and the point of compactification respectively, for $i=1,2$.  The orbifold $(\widehat{X}_i,[\widehat{g}_i])$ has one singularity with orbifold group $\overline{\Gamma_i}$.  Thus, for $i=1,2$, taking the conformal blow-up at this point (since the Yamabe invariant is positive), we obtain a scalar-flat anti-self-dual ALE space $(X_i,g_{X_i})$ with group at infinity $\overline{\Gamma_i}$, satisfying $\pi_1(X_i)=\mathbb{Z}/2\mathbb{Z}$ and $\tau_{top}(X_i)=-k_{(m-n,m)}$.
We note that it follows from the gluing theorem that the second cohomology group of the self-dual deformation complex $H^2_{SD}(\hat{X}_i)$ also vanishes for these spaces. 

Finally, an equivariant version of the gluing theorem can in fact be used to ensure that the spaces $(X_i,g_{X_i})$ admit an isometric $S^1$-action.

\subsection{Proof of Corollary \ref{tor_corollary}}
\label{sequence_corollary_proof}
Let $\Gamma_1=\phi(L(1,2m)\times D^*_{4n})$ and $\Gamma_2=\mathfrak{I}^2_{m,n}$ with $m,n$ as specified in Table \ref{groups} so the action on $S^3$ is free.  For $i=1,2$, let $(Y_i,g_{Y_i})$ denote the scalar-flat K\"ahler, hence anti-self-dual, ALE space with group at infinity $\Gamma_i$,
obtained for the non-cyclic ($n>1$) and cyclic ($n=1$) cases respectively in \cite{LVsmorgasbord,CalderbankSinger}, and let $(X_i,g_{X_i})$ be the scalar-flat anti-self-dual ALE spaces with the orientation reversed groups at infinity $\overline{\Gamma_1}$ and $\overline{\Gamma_2}$ of
Theorem~\ref{tor}.
Taking the conformal compactifications $(\widehat{Y}_i,\widehat{g}_{Y_i})$ and $(\widehat{X}_i,\widehat{g}_{X_i})$ with the self-dual orientation, as in Remark~\ref{compactification_groups}, we obtain self-dual orbifolds with orbifold groups $\Gamma_i$ and $\overline{\Gamma_i}$ respectively.   Since these are orientation reversed conjugate, and because $H^2_{SD}$ vanishes for each orbifold as pointed out above, we are once again able to use the self-dual orbifold gluing of Theorem~\ref{SD_gluing} to obtain a self-dual conformal structure on the orbifold connect sum $\widehat{X}_i\#\widehat{Y}_i$, where the connect sum is taken at the points of compactification.  

Clearly the signature of this space, $\tau(\widehat{X}_i\#\widehat{Y}_i)=-\tau(X_i)-\tau(Y_i)$,
is dependent upon $m$ and $n$, and to highlight this, we denote it by $\ell_i(m,n)$, for $i=1,2$.  For any $n$, $\tau(X_i)$ is easy to find from the work of Section \ref{proof_reversedALE_thm}, and in the non-cyclic case ($n>1$), $\tau(Y_i)$ is given in \cite{LVsmorgasbord}.  In the cyclic case ($n=1$) however, one finds that both $\Gamma_1$ and $\Gamma_2$ are orientation preserving conjugate to $L(2m+1, 4m)$ for their respective $m$, and hence that $\tau(Y_i)=-k_{(2m+1, 4m)}=-3$.  Therefore, we find that
\begin{align}
\label{tau_connect_sum2}
\ell_i(m,n):=\tau(\widehat{X}_i\#\widehat{Y}_i)=3+\begin{cases}
k_{(n-m,n)}+k_{(m-n,m)}\phantom{=}&\text{$n>1$ and $m>1$}\\
m-1\phantom{=}&\text{$n=1$ and $m>1$}\\
n-1\phantom{=}&\text{$n>1$ and $m=1$}\\
0\phantom{=}&\text{$n=1$ and $m=1$}.
\end{cases}
\end{align}
Note that since $m$ must be odd for $\Gamma_2$, the last two cases in \eqref{tau_connect_sum2} cannot occur for this group.  Also, we distinguish
the $m=n=1$ case for $\Gamma_1$ as $(\ell_1(1,1)=3)\#\CP^2$ is the minimal number of $\CP^2s$ on which a self-dual metric is obtained by this technique.

Now, we will show that $\widehat{X}_i\#\widehat{Y}_i$ is, in fact, simply connected.  Cover $\widehat{X}_i\#\widehat{Y}_i$ with open sets $U_{\widehat{X}_i}$ and $U_{\widehat{Y}_i}$ containing the $\widehat{X}_i$ and $\widehat{Y}_i$ components of the connect sum respectively, and so that $U_{\widehat{X}_i}\cap U_{\widehat{Y}_i}$ deformation retracts to $S^3/\Gamma_i$.  Recall that we have the homomorphisms of fundamental groups $i_{\widehat{X}_i}:\pi_1(U_{\widehat{X}_i}\cap U_{\widehat{Y}_i})\rightarrow \pi_1(U_{\widehat{X}_i})$ and $i_{\widehat{Y}_i}:\pi_1(U_{\widehat{X}_i}\cap U_{\widehat{Y}_i})\rightarrow \pi_1(U_{\widehat{Y}_i})$ induced from the respective inclusion maps.  
Observe that $i_{\widehat{X}_i}$ is surjective since $\pi_1(U_{\widehat{X}_i})=\mathbb{Z}/2\mathbb{Z}$ results from the antipodal map on $S^2$ induced by the action of a generator of $\overline{\Gamma_i}$ under the Hopf map.  Also, the map $i_{\widehat{Y}_i}$ is clearly trivial since $\pi_1(U_{\widehat{Y}_i})=\{1\}$.
By the Seifert-van Kampen theorem
\begin{align}
\pi_1(\widehat{X}_i\#\widehat{Y}_i)=\pi_1(U_{\widehat{X}_i})*\pi_1(U_{\widehat{Y}_i})/N,
\end{align}
where $*$ denotes the free product and $N$ is the normal subgroup of $\pi_1(U_{\widehat{X}_i})*\pi_1(U_{\widehat{Y}_i})$ generated by $i_X(\gamma)i_Y(\gamma)^{-1}$ for all $\gamma\in \pi_1(U_{\widehat{X}_i}\cap U_{\widehat{Y}_i})$.  Therefore, given that $i_{\widehat{X}_i}$ is surjective, $i_{\widehat{Y}_i}$ is trivial, and $\pi_1(U_{\widehat{Y}_i})=\{1\}$, as discussed above, we find that $\pi_1(\widehat{X}_i\#\widehat{Y}_i)=~\{1\}$.
By the results of Donaldson-Freedman, 
$\widehat{X}_i\#\widehat{Y}_i$ is homeomorphic to $\ell_i(m,n)\#\CP^2$ (see for example \cite{FU}).  
Finally, by taking sequences of conformal factors which uniformly scale one of the factors to have zero volume in the limit, the orbifold limiting statement follows immediately from this construction.  

Again, an equivariant version of the gluing theorem can in fact be used to ensure that these examples admit a conformally isometric $S^1$-action.

\begin{remark}{\em
The Donaldson-Freedman result only provides a homeomorphism, but
we suspect that these manifolds are in fact {\textit{diffeomorphic}} to $\ell\#\CP^2$.
}
\end{remark}

\bibliography{SD_Quotients_references}

\end{document}